\documentclass[12pt,reqno,a4paper,twoside]{article}
\usepackage{amsmath,amsthm,amstext,amscd,amssymb,euscript,color}
\usepackage{epsf}
\renewcommand{\phi}{\varphi}

\renewcommand{\subset}{\subseteq}
\renewcommand{\emptyset}{\varnothing}

\newcommand{\Zd}{\mathbb{Z}^d}

\renewcommand{\Pr}{\mathbb P}

\def\1{ {\mathit{1} \!\!\>\!\! I} }

\makeatletter
\@addtoreset{equation}{section}
\makeatother

\newtheorem{theorem}{Theorem}
\numberwithin{theorem}{section}
\newtheorem{lemma}{Lemma}
\numberwithin{lemma}{section}
\newtheorem{proposition}{Proposition}
\numberwithin{proposition}{section}
\newtheorem{corollary}{Corollary}
\numberwithin{corollary}{section}
\newtheorem{definition}{Definition}
\numberwithin{definition}{section}
\newtheorem{remark}{Remark}
\numberwithin{remark}{section}
\newtheorem{open}{Open question}
\numberwithin{open}{section}

\newcommand{\beq}{\begin{eqnarray}}
\newcommand{\eeq}{\end{eqnarray}}

\newcommand{\be}{\begin{equation}}
\newcommand{\ee}{\end{equation}}

\newcommand{\bl}{\begin{lemma}}
\newcommand{\el}{\end{lemma}}

\newcommand{\br}{\begin{remark}}
\newcommand{\er}{\end{remark}}

\newcommand{\bt}{\begin{theorem}}
\newcommand{\et}{\end{theorem}}

\newcommand{\bd}{\begin{definition}}
\newcommand{\ed}{\end{definition}}

\newcommand{\bp}{\begin{proposition}}
\newcommand{\ep}{\end{proposition}}

\newcommand{\bc}{\begin{corollary}}
\newcommand{\ec}{\end{corollary}}

\newcommand{\bpr}{\begin{proof}}
\newcommand{\epr}{\end{proof}}

\newcommand{\bi}{\begin{itemize}}
\newcommand{\ei}{\end{itemize}}

\newcommand{\ben}{\begin{enumerate}}
\newcommand{\een}{\end{enumerate}}

\newcommand{\eqn}[2]{\begin{equation}\label{#1}#2\end{equation}}
\newcommand{\eqnst}[1]{\begin{equation*}#1\end{equation*}}
\newcommand{\eqnspl}[2]{\begin{equation}\begin{split}\label{#1}%
    #2\end{split}\end{equation}}
\newcommand{\eqnsplst}[1]{\begin{equation*}\begin{split}%
    #1\end{split}\end{equation*}}

\newcommand{\Z}{\mathbb Z}
\newcommand{\R}{\mathbb R}

\def\Vol{\mathrm{Vol}}

\def\now{
\ifnum\time<60
          12:\ifnum\time<10 0\fi\number\time am
          \else
            \ifnum\time>719\chardef\a=`p\else\chardef\a=`a\fi
          \hour=\time
          \minute=\time
          \divide\hour by 60
          \ifnum\hour>12\advance\hour by -12\advance\minute by-720 \fi
          \number\hour:%
          \multiply\hour by 60
          \advance\minute by -\hour
          \ifnum\minute<10 0\fi\number\minute\a m\fi}
\newcount\hour
\newcount\minute

\numberwithin{equation}{section}

\theoremstyle{remark}

\def\Av{{\mathrm{Av}}}

\def\es{\emptyset}

\def\LE{\mathrm{LE}}

\def\bp{\overline{p}}

\def\last{\mathrm{last}}

\newcommand{\caD}{{\mathcal D}}

\newcommand{\caH}{{\mathcal H}}

\newcommand{\caR}{{\mathcal R}}

\begin{document}

\title{{\bf Approaching criticality via the zero dissipation limit in the abelian avalanche model}}

\author{
Antal A.~J\'arai
\footnote{Department of Mathematical Sciences,
University of Bath,
Claverton Down, Bath, BA2 7AY, United Kingdom,
Email: {\tt A.Jarai@bath.ac.uk}}\\
Frank Redig
\footnote{Delft Institute of Applied Mathematics, Technische Universiteit Delft,
 Mekelweg 4, 2628 CD Delft, Nederland,
Email: {\tt F.H.J.Redig@tudelft.nl}}\\
Ellen Saada
\footnote{CNRS, UMR 8145, MAP5, Universit\'{e} Paris Descartes,
Sorbonne Paris Cit\'{e},
45 rue des Saints-P\`eres,
75270 Paris cedex 06, France,
Email: {\tt Ellen.Saada@mi.parisdescartes.fr}}
}

\maketitle

\footnotesize
\begin{quote}
{\bf Abstract}:
The discrete height abelian sandpile model was introduced
by Bak, Tang \& Wiesenfeld  and Dhar 
as an example for the concept of
self-organized criticality. When the
model is modified to allow grains to disappear on each toppling,
it is called bulk-dissipative. 
We provide a detailed study of a continuous
height version of the abelian sandpile model,
called the abelian avalanche model, which allows an
arbitrarily small amount of dissipation to take place on every
toppling. We prove that for non-zero dissipation, the
infinite volume limit of the stationary measure
of
the abelian avalanche model exists and can be
obtained via a weighted spanning tree measure.
We show that in the whole non-zero dissipation regime, the model
is not critical, i.e., spatial covariances of local observables decay exponentially.
We then study the zero dissipation limit and
prove that the self-organized critical model
is recovered, both for the stationary measure and
for the dynamics.
We obtain rigorous bounds on toppling probabilities and
introduce an exponent describing their scaling at
criticality. We rigorously establish the mean-field value
of this exponent for $d > 4$.
\end{quote}
\normalsize

\vspace{12pt}

\newpage
\tableofcontents
\section{Introduction}

The {\em abelian sandpile model}, introduced by Bak, Tang \&
Wiesenfeld \cite{btw}, is a paradigm of the phenomenon
of {\em self-organized criticality}.
Its fundamental properties, including the
abelianness, the set of recurrent configurations,
and the group of toppling operators, was discovered by Dhar \cite{Dhar}.

Just as for ordinary critical models in statistical mechanics, such as the Ising model or percolation, it is useful to add to the model
a ``magnetic field'' (or a ``mass'') in order to tune it away from criticality, and in such a way that
when the field becomes very small, new results can be obtained for the original critical model, for instance critical exponents,
decay of covariances or percolation properties.

In the abelian sandpile model, a procedure to tune away from criticality is to introduce ``bulk dissipation'',  i.e., to impose that
in the bulk, upon each toppling mass is lost (whereas for the critical model this happens only at the boundary).
Because to go back to criticality we want to tune
this amount of lost mass to zero, we need a version of the abelian sandpile model with continuous heights.
There are at least two reasons for considering such models. On the one hand,
useful information can be gained about the critical model by letting the
bulk dissipation go to zero; see for example \cite{mru}.
On the other hand, with applications on random networks in mind,
the underlying graph may have no natural boundary, but instead,
dissipation can be present at all or some of the vertices.

A continuous model with deterministic additions, called the {\em abelian avalanche model}, has been introduced on finite graphs by Gabrielov \cite{G93}.
We study a stochastic variant of this model,  which
has a dissipation parameter $\gamma\geq 0$, and we still call it
the abelian avalanche model. For $\gamma=0$ it is the natural continuous
analogue of the (critical) abelian sandpile model. Moreover, this model generalizes
the discrete dissipative models with integer dissipation
studied by Maes, Redig \& Saada  in \cite{mrs2}.

\medbreak

The main results of our paper are summarized as follows.
\begin{enumerate}
\item {\em Existence of infinite-volume limits for all $\gamma>0$.}
We prove the existence of the infinite-volume limit for the stationary measures and for the stochastic dynamics
of the
finite-volume abelian avalanche model.
\item {\em Non-criticality for every positive dissipation.}
For every positive value of the dissipation,
the 
infinite-volume model is not critical, i.e.,
covariances of local functions as well as avalanche sizes decay exponentially.
This also extends the results in
\cite{mrs2} for arbitrary small dissipation.
\item  {\em Criticality for $\gamma\to 0$ and the toppling probability exponent.}
We prove that for $\gamma\to 0$ the critical model is recovered, both for the stationary
measures and for the dynamics. We prove lower bounds on the probability
that vertex $x$ topples in the dissipative model, and achieve a sharp lower
bound when $d > 4$. In particular, we introduce the
{\em toppling probability exponent} that describes the scaling
of the above probability at criticality, and establish rigorously
its mean-field value above the conjectured upper
critical dimension, i.e.~for $d > 4$. 
\end{enumerate}

\section{Overview of the models and outline of the paper}
\subsection{Standard (discrete) abelian sandpile model}
Let us start by briefly recalling the standard (i.e., discrete) abelian sandpile model.  Let
$\Lambda \subseteq \mathbb{Z}^d$ be a finite set. A \emph{configuration}
on $\Lambda$ is a collection of particles occupying the sites
in $\Lambda$, specified by a map
$\eta : \Lambda \to \{ 0, 1, \dots \}$.
If $\eta_x \ge 2d$ for $x\in\Lambda$, $x$ is allowed to
\emph{topple}, that is to send one particle along each edge
incident to $x$ in $\mathbb{Z}^d$. Particles reaching
$\mathbb{Z}^d \setminus \Lambda$ are \textit{lost}, i.e. disappear.
We say that $\eta$ is \emph{stable} if no site can topple,
that is $\eta_x < 2d$ for all $x \in \Lambda$.

We define a Markov chain on the set of stable configurations
as follows. At each time step, we add a particle to a stable
configuration $\eta$ at a
randomly chosen site in $\Lambda$, then carry out all possible
topplings (this succession of topplings is called an \textit{avalanche}
and its final result on the configuration is called stabilization)
until a new stable configuration $\eta'$ is reached. Going from $\eta$ to
$\eta'$ is then a single transition in the Markov chain. It was shown
by Dhar \cite{Dhar} that the resulting stable configuration
does not depend on the order of topplings (hence the name
``abelian''), and that the stationary distribution, denoted by
$\nu_\Lambda$, is unique and uniformly distributed on all
recurrent states.

The abelian sandpile model is \emph{critical}, in the sense that
correlations in the stationary measure have \textit{power law
decay} (see \cite{md2} by Majumdar \& Dhar): for all $d \ge 2$ there is a constant
$c = c(d)$ such that
\begin{equation*}
 \lim_{\Lambda \uparrow \mathbb{Z}^d} \left\{
    \nu_\Lambda ( \eta_0 = 0,\, \eta_x = 0 )
    - \nu_\Lambda ( \eta_0 = 0 ) \nu_\Lambda (\eta_x = 0 ) \right\}
   = c |x|^{-2d}(1+o(1)), \quad \text{as $|x| \to \infty$,}
   \end{equation*}
where $|x|$ denotes the Euclidean norm of $x$, and where $o(1)$ goes to zero as $|x|\to\infty$. 

\subsection{Discrete dissipative sandpile model}
The following  modification  of the abelian sandpile  model,
called the \emph{discrete dissipative sandpile model} was introduced in \cite{japon}
by Tsuchiya \& Katori and studied rigorously
by Maes, Redig \& Saada in \cite{mrs2}.  Let $\gamma \ge 1$ be {\em an integer},
and call a configuration \emph{$\gamma$-stable} if
$\eta_x < 2d + \gamma$ for all $x\in\Lambda$. A site $x\in \Lambda$ is allowed to
\emph{$\gamma$-topple} if $\eta_x \ge 2d + \gamma$, which means that
 it sends $1$ particle along each edge incident to $x$ in $\mathbb{Z}^d$
\textit{and} $\gamma$ particles are lost, or \emph{dissipated}, thereby
decreasing the number of particles at $x$ by $2d + \gamma$.
The presence of dissipation introduces exponential decay of
avalanche sizes, hence the model becomes non-critical (for {\em large values of $\gamma$} it was
shown in \cite{mrs2} that also covariances of local observables decay exponentially). 
The exponential decay of the massive Green's function enabled  in \cite{mrs2} to
define the dissipative sandpile model on $\mathbb{Z}^d$, and to extend
Dhar's formalism of the sandpile group (\cite{Dhar}) to the infinite case.

\subsection{Continuous dissipation}
Mahieu \& Ruelle \cite{mru} considered the limit $\gamma \downarrow 0$ in
analytical expressions obtained for stationary expectation of special observables
in the discrete dissipative sandpile model,
these making sense for any real $\gamma > 0$.  It is a
natural question whether the limit $\gamma \downarrow 0$
can be made sense of in an appropriate continuous model.
The problem of properly
defining this model was circumvented in the physics literature,
where only specific correlation functions involving the Green's
function were considered, within which then the massive Green's function can be substituted.
Our goal in this paper is to define and establish rigorously some of the basic
properties of such a \textit{continuous height} dissipative sandpile
model and to use these to derive properties of the critical model.

Let us now briefly introduce this  \textit{abelian
avalanche model}. It is a variant
of a model introduced by Gabrielov \cite{G93}, in which the
system is driven deterministically; 
however, as we make use of it
later, the main results of \cite{G93} apply to the stochastic case. 

By a (continuous) \emph{configuration} we mean a map
$\eta : \Lambda \to [0,\infty)$, referred to as a collection
of \emph{heights}. Let $\gamma \ge 0$ be a
\textit{real} parameter. We say that $\eta$ is
\emph{$\gamma$-stable} if $\eta_x \in [0, 2d + \gamma)$
for all $x \in \Lambda$. We say that $x$ is \emph{allowed to
$\gamma$-topple} if $\eta_x \ge 2d + \gamma$. In this
case a $\gamma$-toppling at $x$ means that height $1$ is sent
along each edge incident to $x$ in $\mathbb{Z}^d$, and height
$\gamma$ is lost, thereby decreasing the height at $x$
by $2d + \gamma$. As in the discrete case it holds,
by essentially the same proof, that any configuration
has a unique $\gamma$-stabilization arrived at by carrying out
all possible $\gamma$-topplings.

We define a pure jump Markov process on the set of $\gamma$-stable continuous
configurations as follows. Let $\{ \varphi(x) \}_{x \in \Lambda}$
be a collection of positive rates. For $x \in \Lambda$, height $1$ is added at
vertex $x$ according to
a Poisson process with rate $\varphi(x)$ (Poisson
processes associated to different sites
are mutually independent). After
an addition of height $1$ has occurred, the configuration is instantaneously
$\gamma$-stabilized by carrying out all possible $\gamma$-topplings.
By analogy with \cite{G93}, we call this dynamics the abelian avalanche model.

\subsection{Outline of the paper}
Section \ref{sec:model} collects basic properties of the
finite volume abelian avalanche model: In Section \ref{ssec:defns} we give
the precise definitions of $\gamma$-toppling matrices, $\gamma$-stabilization
and $\gamma$-toppling numbers. We discuss properties of the set
$\mathcal{R}_\Lambda$ of ``recurrent'' configurations in
Section \ref{ssec:stat}. The natural invariant measure of this model is the
normalized Lebesgue measure on $\mathcal{R}_\Lambda$,
denoted by $m^{(\gamma)}_\Lambda$.

In Section \ref{ssec:rational} we consider rational $\gamma = k/n$.
In that case the fractional part of the vector
$(n \eta_x)_{x \in \Lambda}$ is invariant under the dynamics.
When this is factored out,
one obtains the discrete dissipative sandpile model,
with an ergodic dynamics.

Due to the toppling rule, our continuous model retains some
discreteness, for all values of $\gamma$.
In fact, continuous heights will only be used to define the dynamics;
the stationary measure has a natural description in terms of appropriate
discrete height variables, that we define in Section \ref{ssec:trees}.
Those discretized heights allow us to adapt the burning
bijection established by Majumdar \& Dhar in \cite{MD92}, and to
define a coding of those configurations with discretized heights
in terms of weighted spanning trees. The coding by weighted spanning
trees is extended to waves (an avalanche is decomposed in a succession
of \textit{waves}) in Section \ref{ssec:waves}.

In Section \ref{ssec:ergod} we discuss ergodic properties of
the finite volume abelian avalanche 
model. We show that when
$\gamma$ is irrational, the stationary process started from
the invariant measure $m^{(\gamma)}_\Lambda$ is ergodic in time.
Another natural question is related to the following
transformation: add unit height at a fixed vertex, and
$\gamma$-stabilize. This is a measure preserving transformation
on $(\mathcal{R}_\Lambda, m^{(\gamma)}_\Lambda)$. We give a criterion for
the transformation to be ergodic, and provide examples
when the criterion can be verified.

The weak limit of the stationary measure
$m^{(\gamma)} = \lim_{\Lambda \uparrow \mathbb{Z}^d} m^{(\gamma)}_\Lambda$
is constructed in Section \ref{sec:stat-infinite},
using the burning bijection and Wilson's algorithm \cite{Wilson}.
We prove exponential decay of correlations of
local functions, for all $\gamma > 0$, in Theorem \ref{ijsbeer}.
As $\gamma \downarrow 0$, the (discretized) stationary
measure on $\mathbb{Z}^d$ converges to the critical sandpile
measure; this is shown in Theorem \ref{lem:zerolim}.
For certain special local events an estimate on
the rate of convergence is given in
Proposition \ref{prop:rate}.

The abelian avalanche model
on $\mathbb{Z}^d$ is considered in Section \ref{sec:dynamics}.
Theorem \ref{thm:dissip-dyn}
proves the existence of a sandpile Markov process
on $\mathbb{Z}^d$ with dissipation. In section \ref{avtail}, Theorem \ref{thm:waves}
gives estimates on
toppling probabilities 
under $m^{(\gamma)}$,
and Theorem \ref{thm:new-av-lb}
 establishes the mean-field value of the toppling
probability exponent in dimensions $d > 4$. 

In Section \ref{sec:zero-dissip-lim} we prove that
the abelian avalanche model
on $\mathbb{Z}^d$ converges, as $\gamma \downarrow 0$, to the critical (non-dissipative)
abelian sandpile process, when the latter is known to exist on $\mathbb{Z}^d$;
see Theorem \ref{thm:converge}.

Our models will live on $\mathbb{Z}^d$, and for simplicity, we only consider
the nearest neighbor case.
 Presumably one can extend our
constructions to other infinite graphs on which each tree of the wired uniform
spanning forest has one end; see \cite{Jarai10b,JW12}.

\section{The abelian avalanche model in finite volume}
\label{sec:model}
\subsection{Toppling matrices, stabilization, toppling numbers}
\label{ssec:defns}
For $x,y \in \mathbb{Z}^d$, we write $|x-y|$ for their Euclidean
distance, and
$\| x-y\|_1 := \sum_{i=1}^d |x_i-y_i|$. We denote $x \sim y$ if
$x$ and $y$ are neighbors, that is $|x - y| = 1$.
Each site $x$ will carry a \textit{continuous} height variable with value
in the interval $[0,2d + \gamma)$, where $\gamma \ge 0$ is a
\textit{real} parameter.
On a $\gamma$-\textit{toppling}, a site will give height $1$ to each
of its neighbors, and lose height $2d + \gamma$, that is,
an amount $\gamma$ of height is dissipated on each toppling
(when it is clear in context, we sometimes write toppling instead of $\gamma$-toppling). 
This can be summarized using the $\gamma$-\emph{toppling matrix} $\Delta^{(\gamma)}$ with elements 
\begin{equation*}
  \Delta^{(\gamma)}_{xy}
  = \begin{cases}
    2d + \gamma & \text{if $x = y$;}\\
    -1          & \text{if $x \sim y$;}\\
    0           & \text{otherwise.}
    \end{cases}
    \end{equation*}
For $\Lambda \subseteq \mathbb{Z}^d$, we write
$\Delta^{(\gamma)}_\Lambda = ( \Delta^{(\gamma)}_{xy} )_{x,y \in \Lambda}$
for the $\gamma$-toppling matrix restricted to $\Lambda$.

We define the sets of \emph{(continuous) height configurations}
in $\mathbb{Z}^d$ and in $\Lambda$ by
\begin{equation*}
 {\mathcal X}
  = [0, \infty)^{\mathbb{Z}^d} \qquad \text{and}\qquad
  {\mathcal X}_\Lambda
  = [0, \infty)^\Lambda.
  \end{equation*}
A site $x$ is called $\gamma$-{\sl stable} in configuration
$\eta$, if $\eta_x < \Delta^{(\gamma)}_{xx} = 2d + \gamma$;
otherwise, $x$ is $\gamma$-{\sl unstable} in configuration $\eta$.
The sets of $\gamma$-stable configurations are denoted by
\begin{equation*}
\Omega^{(\gamma)}
  = [0, 2d + \gamma)^{\mathbb{Z}^d} \qquad \text{and}\qquad
  \Omega^{(\gamma)}_\Lambda
  = [0, 2d + \gamma)^\Lambda.
  \end{equation*}
Sometimes we will need configurations that are only stable
in $\Lambda$ and are unrestricted in $\mathbb{Z}^d \setminus \Lambda$,
so we will also use
\begin{equation*}
 \overline{\Omega}^{(\gamma)}_\Lambda
  = \{ \eta \in {\mathcal X} : \eta_x < 2d + \gamma,\, x \in \Lambda \}.
  \end{equation*}
A $\gamma$-{\sl toppling} of site $x$ in $\Lambda$ is defined
by the operator $T^{(\gamma)}_{\Lambda,x}$,
via the formula
\begin{equation}\label{top}
  \left(T^{(\gamma)}_{\Lambda,x} \eta \right)_y
  = \eta_y - \Delta^{(\gamma)}_{xy}, \qquad
    y \in \Lambda, \eta \in {\mathcal X}_\Lambda.
\end{equation}
We also define $T^{(\gamma)}_{\Lambda,x} \eta$, when $\eta \in {\mathcal X}$.
In this case, the values of $\eta_y$ for $y \in \mathbb{Z}^d \setminus \Lambda$
are left unchanged, and for $y \in \Lambda$ they change
according to \eqref{top}.

 A $\gamma$-toppling is called \emph{legal} if the toppled
site $x\in\Lambda$ is $\gamma$-unstable before toppling, which ensures
that each component of $T^{(\gamma)}_{\Lambda,x} \eta$ is still
non-negative. A sequence of $\gamma$-topplings at sites
$(x_1,\ldots,x_n)\in\Lambda^n$ is called \emph{$(\Lambda,\gamma)$-stabilizing}
for $\eta \in {\mathcal X}_\Lambda$, if each toppling can be carried out and
the final result is $\gamma$-stable; that is, if
\begin{itemize}
\item[{\em (i)}] $T^{(\gamma)}_{\Lambda,x_k}$ is a legal $\gamma$-toppling of
 $T^{(\gamma)}_{\Lambda,x_{k-1}} \circ \dots \circ
 T^{(\gamma)}_{\Lambda,x_1} \eta$,
 $1 \le k \le n$;
\item[{\em (ii)}] the final configuration is
in $\Omega^{(\gamma)}_\Lambda$.
\end{itemize}
Note that for all finite $\Lambda\subseteq\mathbb{Z}^d$ and all
$\eta \in {\mathcal X}_\Lambda$, a $(\Lambda,\gamma)$-stabilizing sequence exists.
The number of times a site
topples does not depend on the
$(\Lambda,\gamma)$-stabilizing sequence, and hence there
is a well-defined $\gamma$-\emph{stabilization map}
${\mathcal S}^{(\gamma)}_\Lambda : {\mathcal X}_\Lambda \to \Omega^{(\gamma)}_\Lambda$,
see e.g.\  \cite{Dhar}, or \cite[Appendix B]{G93} for a proof.
The succession of
topplings in a $\gamma$-stabilization
is called a $\gamma$-avalanche.

The result of $\gamma$-stabilization is related to the original
configuration as follows. Let $N^{(\gamma)}_\Lambda(\eta)$ be the
vector consisting of the $\gamma$-\emph{toppling numbers} associated
to the $(\Lambda,\gamma)$-stabilization of $\eta$, i.e.,
$(N^{(\gamma)}_\Lambda(\eta))_x$ is the number of times $x \in \Lambda$ 
topples during $\gamma$-stabilization of $\eta$ in $\Lambda$.
Then by \eqref{top}, the net effect of all 
topplings can be written as
\begin{equation}\label{e:stabilize}
 ({\mathcal S}^{(\gamma)}_\Lambda \eta)_y
  = \eta_y - \sum_{x \in \Lambda}
    \Delta^{(\gamma)}_{yx}(N^{(\gamma)}_\Lambda(\eta))_x
  = \eta_y - (\Delta^{(\gamma)}_\Lambda N^{(\gamma)}_\Lambda(\eta) )_y,
    \qquad y \in \Lambda.
    \end{equation}  
Similarly, $\gamma$-stabilization in volume $\Lambda$ can also be viewed as a
map ${\mathcal S}^{(\gamma)}_\Lambda: {\mathcal X} \to \overline{\Omega}^{(\gamma)}_\Lambda$,
where the coordinates outside $\Lambda$ are left unchanged.

The \emph{addition operators} are defined by
adding height $1$ at a site $x$, and then $\gamma$-stabilizing:
\begin{equation}\label{ax}
 a^{(\gamma)}_{\Lambda,x}
    : \Omega^{(\gamma)}_\Lambda \to \Omega^{(\gamma)}_\Lambda ;\qquad\qquad
  a^{(\gamma)}_{\Lambda,x} \eta
  = {\mathcal S}^{(\gamma)}_\Lambda (\eta + \delta_x),
  \end{equation}
where $\delta_x$ denotes the vector having entry equal
to one at site $x$ and zero elsewhere.
The addition operators commute, that is, for all
$x,y\in\Lambda$, $\eta\in\Omega^{(\gamma)}_\Lambda$,
\begin{equation}\label{com}
 a^{(\gamma)}_{\Lambda,x} ( a^{(\gamma)}_{\Lambda,y} \eta)
 = a^{(\gamma)}_{\Lambda,y} (a^{(\gamma)}_{\Lambda,x} \eta).
\end{equation}
This follows from the fact that $\gamma$-stabilization
is well-defined: indeed both expressions in \eqref{com}
are equal to ${\mathcal S}^{(\gamma)}_\Lambda (\eta + \delta_x + \delta_y)$.

We endow ${\mathcal X}$, ${\mathcal X}_\Lambda$, $\Omega$, $\Omega_\Lambda$
with the product metric
\begin{equation}\label{e:dist}
 {\mathrm{dist}}(\eta_1,\eta_2)
  = \sum_{x} 2^{-|x|} \min\{ |(\eta_1)_x - (\eta_2)_x|, 1 \},
  \end{equation}
where the sum is over $\mathbb{Z}^d$ or over $\Lambda$.

Given a function ${\varphi}: \Lambda \to (0,\infty)$, we define a
jump Markov process on $\gamma$-stable configurations.
The action of the generator on Borel measurable functions
$f : \Omega^{(\gamma)}_\Lambda \to {\mathbb R}$ is given by
\begin{equation}\label{gendis}
 {\mathcal{L}}_\Lambda f(\eta)
 = {\mathcal{L}}^{\gamma,{\varphi}}_\Lambda f(\eta)
 = \sum_{x\in\Lambda} {\varphi}(x)
   \left[ f \left( a^{(\gamma)}_{\Lambda,x} \eta \right)
   - f(\eta) \right].
\end{equation}
The above process is described in words as follows:
at each site $x \in \Lambda$ we have a Poisson process with
intensity ${\varphi}(x)$, and at different sites these
processes are independent. At the event times of this
Poisson process we apply the addition operator
$a^{(\gamma)}_{\Lambda,x}$ to the configuration.

 Given a measure $\mu$, we denote by ${\mathbb E}_\mu$ expectation with respect to
 $\mu$. 

\subsection{Stationary measure, Dhar's formula}
\label{ssec:stat}
As in the discrete case \cite{Dhar}, there is a subset
$\mathcal{R}^{(\gamma)}_\Lambda \subseteq \Omega^{(\gamma)}_\Lambda$, called
the set of ``recurrent configurations'', such that
any invariant measure is concentrated on $\mathcal{R}^{(\gamma)}_\Lambda$.
This is described via the notion of $\gamma$-\emph{allowed configuration}
defined below. A \emph{$\gamma$-forbidden subconfiguration} ($\gamma$-FSC)
is a pair $(W,\eta_W)$ where  $\varnothing \not= W \subseteq \mathbb{Z}^d$  is finite,
$\eta_W \in {\mathcal X}_W$, such that for all $y \in W$,
the number of neighbors of $y$ in $W$ is strictly larger than the height
 $(\eta_W)_y$:
\begin{equation}\label{fscg}
  (\eta_W)_y
  < \sum_{z: z \in W,\, z \not= y} (-\Delta^{(\gamma)}_{zy}).
\end{equation}
\begin{definition}\label{def:allowed}
A configuration $\eta\in\Omega^{(\gamma)}_\Lambda$
(respectively $\eta \in \Omega^{(\gamma)}$)
is called
$\gamma$-{\rm allowed}
if there does not exist finite $W\subseteq\Lambda$ (respectively
$W \subseteq \mathbb{Z}^d$) such that
the pair consisting of $W$ and the restriction $\eta_W$ of $\eta$
to $W$ is a $\gamma$-FSC.
\end{definition} 
\begin{remark}
\label{rem:FSC}
Since the off-diagonal elements of the $\gamma$-toppling matrix
$\Delta^{(\gamma)}$ do not depend on $\gamma$, the right hand
side of inequality \eqref{fscg} is independent of $\gamma$.
Therefore we have the same forbidden subconfigurations for
any value of $\gamma$.  So from now on we use the words FSC
rather than $\gamma$-FSC, and allowed rather than
$\gamma$-allowed.
\end{remark}
Let
\begin{equation*}\begin{split}
 {\mathcal R}^{(\gamma)}_\Lambda
  &= \{ \eta \in \Omega^{(\gamma)}_\Lambda :
  \text{$\eta$ is allowed} \}, \\
  {\mathcal R}^{(\gamma)}
  &= \{ \eta \in \Omega^{(\gamma)} :
  \text{$\eta$ is allowed} \}:= \{ \eta \in \Omega^{(\gamma)} :
  \text{$\eta_V \in {\mathcal R}^{(\gamma)}_V$
  for all finite $V \subseteq \mathbb{Z}^d$} \}.
  \end{split}\end{equation*}
The results of \cite[Sections 3,4]{G93} imply the
following properties of allowed configurations.
 We write $\mathrm{Vol}$ for Lebesgue measure
on $\R^\Lambda$. 
\begin{proposition}\label{gabrielov}
\begin{itemize}
\item[(i)] The addition operator $a^{(\gamma)}_{\Lambda,x}$ maps
  ${\mathcal R}^{(\gamma)}_\Lambda$ one-to-one and onto itself.
\item[(ii)] ${\mathrm{Vol}}({\mathcal R}^{(\gamma)}_\Lambda)
  = \det(\Delta^{(\gamma)}_\Lambda)$.
\item[(iii)] Lebesgue measure on ${\mathcal R}^{(\gamma)}_\Lambda$
is invariant under $a^{(\gamma)}_{\Lambda,x}$, $x \in \Lambda$.
\end{itemize}
\end{proposition}
Hence, the probability measure $m^{(\gamma)}_\Lambda$
on ${\mathcal R}^{(\gamma)}_\Lambda$ defined by
\[m^{(\gamma)}_\Lambda(A)
:=\frac{ {\mathrm{Vol}}(A)}{{\mathrm{Vol}}({\mathcal R}^{(\gamma)}_\Lambda)}
\]
is stationary for the Markov process defined in \eqref{gendis} of
Section \ref{ssec:defns}. The set $\mathcal{R}^{(\gamma)}_\Lambda$ is a group
under pointwise addition and $\gamma$-stabilization, isomorphic
to ${\mathbb R}^\Lambda/\Delta^{(\gamma)}_\Lambda$,
the quotient group obtained by identifying
different elements of $\R^\Lambda$ that differ by an integer column multiple
of the matrix $\Delta^{(\gamma)}_\Lambda$.
For $\eta\in\Omega^{(\gamma)}_\Lambda$, we denote by
$n^{(\gamma)}_\Lambda(x,y,\eta)=\left(N_\Lambda^{(\gamma)}(\eta+\delta_x)\right)_y$
the number of $\gamma$-topplings
at $y$ needed to $\gamma$-stabilize $\eta + \delta_x$.
Then due to \eqref{e:stabilize}, we have the relation
\begin{equation}\label{e:addition}
 \left(a^{(\gamma)}_{\Lambda,x}\eta\right)_z
  = \eta_z  + {\bf 1}_{\{x\}}(z)
    - \sum_{y\in\Lambda}
    \Delta^{(\gamma)}_{zy} n^{(\gamma)}_\Lambda(x,y,\eta).
 \end{equation}
Taking expectation with respect to $m^{(\gamma)}_\Lambda$  and using
stationarity under the action of the addition operators 
(that is, Proposition \ref{gabrielov}{\em (iii)})
gives ``Dhar's formula'' \cite{Dhar}:
\begin{equation}
\label{dhar}
  {\mathbb E}_{m^{(\gamma)}_\Lambda} (n^{(\gamma)}_\Lambda(x,y,\eta))
  = (\Delta^{(\gamma)}_\Lambda)^{-1}_{xy}
  =: G^{(\gamma)}_\Lambda (x,y),\qquad z\in\Lambda.
\end{equation}
 If $\gamma > 0$ or $d \ge 3$, the inverse
$G^{(\gamma)}(x,y) := (\Delta^{(\gamma)})^{-1}_{xy}$
also exists and is equal to the limit
$\lim_{\Lambda\uparrow\mathbb{Z}^d} G^{(\gamma)}_\Lambda (x,y)$.
Note that $G^{(\gamma)}(x,y)$ equals the Green's
function of a continuous time random walk that
crosses an edge at rate 1, and is killed at rate
$\gamma$. Similarly, $G^{(\gamma)}_\Lambda(x,y)$
equals the Green's
function of a continuous time random walk that
crosses an edge at rate 1, is killed upon exiting
$\Lambda$ and is killed (inside $\Lambda$) at rate
$\gamma$.

Markov's inequality and \eqref{dhar} imply
\begin{equation}
\label{markest}
  m^{(\gamma)}_\Lambda \left( n^{(\gamma)}_\Lambda(x,y,\eta) \geq 1 \right)
  \leq G^{(\gamma)}_\Lambda (x,y).
\end{equation}
\subsection{Green's function estimates}\label{greenest}
Here we provide estimates
for the Green's function $G^{(\gamma)}(x,y)$. It is well-known
that this Green's function decays
exponentially in the distance to the origin.
We prefer however to insert a proof for the sake
of self-containedness, and to indicate the power of
$\gamma$ entering in the exponent. This will be important later
on when we consider the limit $\gamma\downarrow 0$ at several places.
\begin{lemma}\label{masgreen}
There exist $C > 0$ and $c > 0$ such that for all
$\Lambda\subseteq \mathbb{Z}^d$, $0<\gamma < 1$,
and $x,y\in\mathbb{Z}^d$, $x\not= y$,
\begin{equation}\begin{split}\label{haha}
 G^{(\gamma)}_\Lambda (x,y)
  &\leq G^{(\gamma)} (x,y) \\
  G^{(\gamma)} (x,y)
  &\leq
  \begin{cases}
    \displaystyle{\frac{C \gamma^{d/4 - 1}}{|x-y|^{d/2}}
    e^{ -c \sqrt{\gamma} |x-y| }}
      & \text{if $|x-y| \ge \gamma^{-1/2}$;} \\
   \displaystyle{ \frac{C}{|x-y|^{d-2}} }
      & \text{if $|x-y| \le \gamma^{-1/2}$, $d \ge 3$;} \\
    C+ C \log(|x-y|^{-1} \gamma^{-1/2})
      & \text{if $|x-y| \le \gamma^{-1/2}$, $d = 2$.}
  \end{cases}
  \end{split}\end{equation}
Furthermore, there exist $C'>0$, $c'>0$ such that
the reverse inequalities of \eqref{haha} hold with $C'$ replacing $C$
and $c'$ replacing $c$.
\end{lemma}
\begin{proof}
First note that
$G^{(\gamma)}_\Lambda (x,y) \leq G^{(\gamma)} (x,y)$,
since in defining $G^{(\gamma)}_\Lambda (x,y)$,
the random walk is killed upon exiting $\Lambda$.
Next, we have
\begin{equation}\label{e:Greensum}
 G^{(\gamma)} (x,y)
  = \sum_{n = \|x-y\|_1}^\infty \left(\frac{2d}{2d+\gamma}\right)^n
    p_n(x,y),
\end{equation}
where $p_n(x,y)$ denotes the $n$-step transition probability
of simple (nearest neighbor) random walk $\{ S_n \}$ on $\mathbb{Z}^d$,
and where the sum over $n$ starts at $n = \|x-y\|_1$ since the
nearest neighbor random walk has to make at least
that number of steps to reach $y$ from $x$. Below we write
${\mathbb P}$ for the underlying probability measure.

We use the Gaussian upper and lower bounds:
\begin{equation}\label{e:Gauss1}
 p_n (x,y)
  \leq \frac{C_2}{n^{d/2}} e^{-C_1 |x-y|^2/n},
  \end{equation}
and for $\|x-y\|_1 \leq n $ with $n$ of the same parity
as $\|x-y\|_1$,
\begin{equation}\label{e:Gauss2}
 p_n(x,y) \ge \frac{C_2'}{n^{d/2}} e^{-C_1'|x-y|^2/n}.
 \end{equation}
These are well-known; see \cite{HS-C} for a much more general
result on groups. For the reader's convenience, we supply the sketch of
the proof in the case of $\mathbb{Z}^d$.

The upper bound follows from the large deviation 
bound ${\mathbb P}( |S_m| > |x|/2 ) \le C_5 \exp( - C_1 |x|^2/m)$, 
and the fact that ${\mathbb P} ( S_m = y ) \le C_4 m^{-d/2}$,
a consequence of the local limit theorem \cite{LL10}.
Taking $m = \lfloor n/2 \rfloor$ (where $\lfloor a \rfloor$
denotes the integer part of a real $a$), the two imply:
\begin{equation*}\begin{split}
p_n(0,x)
  &\le \sum_{z : |z| \ge |x|/2}
      p_m(0,z) p_{n-m}(z, x)
      + \sum_{z : |z-x| \ge |x|/2}
      p_m(0,z) p_{n-m}(z,x) \\
  &\le C_2 n^{-d/2} \exp( - C_1 |x|^2/n ).
  \end{split}\end{equation*}
The Gaussian lower bound follows from a chaining argument:
assuming $|x|^2 > n$, let $m = \lfloor |x|^2/n \rfloor$, and
let $y_0 = 0, y_1, \dots, y_m = x$ be points such that
$|y_{i-1} - y_i| \le 2 n / |x|$. Consider the balls
$B_i = B(y_i, n / |x|)$. Then for any $z \in B_{i-1}$,
the central limit theorem implies
${\mathbb P} ( S_{n/m} \in B_i | S_0 = z ) \ge c > 0$. We get,
using the local limit theorem,
\[
p_n(0,x) \ge c^{m-1} C_2' (|x|^2/n^2)^{d/2}
\ge C_2' n^{-d/2} \exp(-C_1' |x|^2/n ).
\]
Inserting the upper estimate \eqref{e:Gauss1} into \eqref{e:Greensum},
and using the notation $C_3 = 1/(2d)$, we have
$2d / (2d + \gamma) \le \exp( - C_3 \gamma )$,
and we obtain
\begin{eqnarray}\label{moervos}
G^{(\gamma)} (0,x)
&\leq &
\sum_{n=1}^\infty \left(\frac{2d}{2d+\gamma}\right)^n
\frac{C_2}{n^{d/2}} e^{-C_1|x|^2/n}
\nonumber\\
&\leq &
  \sum_{n=1}^\infty \frac{C_2}{n^{d/2}}
  \exp \left( - C_3 \gamma n - C_1 \frac{|x|^2}{n} \right).
\end{eqnarray}
In the case $|x| \ge \gamma^{-1/2}$, the bound in \eqref{haha}
follows from estimating separately the sums
\begin{equation}\label{e:twosums}
 \sum_{1 \le n \le |x|/\sqrt{\gamma}}
    \frac{C_2}{n^{d/2}} \exp ( - C_1 |x|^2/n )
    \quad \text{  and  } \,\,
  \sum_{n > |x|/\sqrt{\gamma}}
    \frac{C_2}{n^{d/2}} \exp ( - C_3 \gamma n ).
    \end{equation}
In the case $|x| < \gamma^{-1/2}$, $d = 3$, the bounds
follow by using $2d/(2d + \gamma)\le 1$, and
\cite[Theorem 1.5.4]{Lawler}. In the case $d = 2$
we again estimate the sums \eqref{e:twosums}.

The proof of the lower bound is similar, starting
from the lower bound on $p_n(x,y)$.
\end{proof}

\subsection{Rational $\gamma$}
\label{ssec:rational}
 When $\gamma$ is rational, the abelian avalanche model has a
natural reduction to a discrete dissipative sandpile model that we
now describe. The main results of the paper will not rely on
this section. 
Let $\gamma = k/n$,
with $k \ge 0$, $n \ge 1$ integers, and $k$ and $n$
relatively prime. Then the vector
$(n \eta_x)_{x \in \Lambda}$ changes by integer amounts
both during addition of unit height, and during $\gamma$-toppling.
Hence the fractional part $( \{ n \eta_x \} )_{x \in \Lambda}$
remains invariant, and can be ``factored out''.
We define the map
$ \varphi_\Lambda :
{\mathcal X}_\Lambda \to \{0, 1, 2, \dots \}^\Lambda$,
by $( \varphi_\Lambda(\eta))_x = \lfloor n \eta_x \rfloor$.
The  corresponding  discrete
$\gamma$-toppling matrix  has elements 
\begin{equation*}
 \Delta_{xy}
  = \begin{cases}
    2d n + k &\text{if $x = y$;} \\
    -n       &\text{if $x \sim y$;} \\
    0        &\text{otherwise,}
    \end{cases}
    \end{equation*}
and it has associated addition operators $a_{\Lambda,x}$,
allowed configurations ${\mathcal R}_\Lambda$, and $\gamma$-toppling operators
$T_{\Lambda,x}$. We write $\Delta_\Lambda=(\Delta_{xy})_{x,y\in\Lambda}$
for the $\gamma$-toppling matrix restricted to $\Lambda$.

Notice that $(W, \eta_W)$ is a
$(k/n)$-FSC if and only if $(W,  \varphi_\Lambda(\eta_W))$ is
an FSC with respect to $\Delta$. This implies that
$ \varphi_\Lambda({\mathcal R}^{(k/n)}_\Lambda) = {\mathcal R}_\Lambda$, and
that the stationary measure for the discrete model
coincides with  $ \varphi_\Lambda m^{(k/n)}_\Lambda:=m^{(k/n)}_\Lambda\circ\varphi_{\Lambda}^{-1}$. 
The relation between 
toppling operators is
$T_{\Lambda,x}  \varphi_\Lambda =  \varphi_\Lambda T^{(k/n)}_{\Lambda,x}$.
Consequently, since adding unit height in the
continuous model corresponds to adding $n$ particles
in the discrete model, with the notation
$b_{\Lambda,x} = (a_{\Lambda,x})^n$, we have
$b_{\Lambda,x}  \varphi_\Lambda =  \varphi_\Lambda a^{(k/n)}_{\Lambda,x}$.

The elements $\{ b_{\Lambda,x} \}_{x \in \Lambda}$
generate the sandpile group for $\Delta_\Lambda$.
To see this, note that the order of the group,
$\det(\Delta_\Lambda)$, is relatively prime to $n$.
Hence, powers of $b_{\Lambda,x}$ yield all powers
of $a_{\Lambda,x}$, and the claim follows.
Therefore the reduced (discrete) sandpile model is
ergodic.

\subsection{Discretized heights, burning algorithm and spanning trees}
\label{ssec:trees}
The measure $m^{(\gamma)}_\Lambda$ can be described in terms of discrete height variables.
We introduce the discretizing map
\[
\psi_\Lambda : \Omega^{(\gamma)}_\Lambda \to \Omega^{\mathrm{discr}}_\Lambda:=\{0, 1, \dots, 2d-1, 2d \}^\Lambda
\]
defined by
\begin{equation*}
 \psi_\Lambda(\eta)_y
  = \begin{cases}
    m & \text{if $m \le \eta_y < m + 1$, $m = 0, 1, \dots, 2d-1$;} \\
    2d & \text{if $2d \le \eta_y<\gamma + 2d$.}
    \end{cases}
    \end{equation*}
 We define $\psi : \Omega^{(\gamma)} \to \Omega^{\mathrm{discr}}:=\{0, 1, \dots, 2d-1, 2d \}^{\mathbb{Z}^d}$
analogously.

Notice that the height $2d$ is possible in the discretization of a stable configuration
when $\gamma>0$, whereas when $\gamma=0$ the discretization of a stable configuration has possible heights up to $2d-1$ as in the standard discrete abelian sandpile model.

We define
${\mathcal R}^{\mathrm{discr}}_\Lambda$ to be the set
of configurations $\xi \in \Omega^{\mathrm{discr}}_\Lambda $
which are allowed (cf. Definition \ref{def:allowed}).
By Remark \ref{rem:FSC} and the fact that
the right hand side of \eqref{fscg} is an integer,
for any $\eta \in \Omega^{(\gamma)}_{\Lambda}$, we have
\begin{equation}\label{e:Requiv}
 \eta \in {\mathcal R}^{(\gamma)}_\Lambda
  \qquad \text{if and only if} \qquad
  \psi_\Lambda(\eta) \in {\mathcal R}^{\mathrm{discr}}_\Lambda.
  \end{equation}
By a \emph{$(\gamma,\Lambda)$-cell}, we will mean a subset
of ${\mathcal R}^{(\gamma)}_\Lambda$ of the form
${\mathcal R}^{(\gamma)}_\Lambda \cap \psi_\Lambda^{-1}(\xi)$, for some
$\xi \in {\mathcal R}^{\mathrm{discr}}_\Lambda$.
It follows from the above discussion that
$m^{(\gamma)}_\Lambda$ is uniform on each cell, hence
$m^{(\gamma)}_\Lambda$ can be uniquely specified in terms
of the measures of cells. Let
$\nu^{(\gamma)}_\Lambda := \psi_\Lambda m^{(\gamma)}_\Lambda$.
We proceed to give a description of
$\nu^{(\gamma)}_\Lambda(\xi)$, $\xi \in {\mathcal R}^{\mathrm{discr}}_\Lambda$.
Let
\begin{equation}\label{e:Hxi}
 {\mathcal H}(\xi)
  = |\{ y \in \Lambda : \xi_y = 2d \}|,
  \end{equation}
where $|A|$ denotes the cardinality of a set $A$. Then for $\gamma>0$
\begin{equation}\label{majum}
\nu^{(\gamma)}_\Lambda(\xi)
= \frac{\gamma^{{\mathcal H}(\xi)}}{\det(\Delta^{(\gamma)}_\Lambda)},
\end{equation}
which follows from the fact that under $\psi_\Lambda^{-1}$, discrete heights
$\xi_x\in \{0,\ldots, 2d-1\}$ go to
intervals of unit length, and
heights $\xi_x= 2d$ to
intervals of length $\gamma$, hence
$\mathrm{Vol} (\psi_\Lambda^{-1} (\xi))= \gamma^{{\mathcal H}(\xi)}$.
\begin{remark}
\label{rem:gamma=0}
When $\gamma = 0$, ${\mathcal H}(\xi) = 0$ for all $\xi$, and
$\nu_\Lambda^{(0)}$ is uniform on allowed configurations
 such that all heights are $< 2d$. 
\end{remark}
In order to study the infinite volume limit,
we interpret $\nu^{(\gamma)}_\Lambda(\xi)$ in terms of \textit{weighted
spanning trees}. For this we adapt to our setting the
\textit{burning bijection} of Majumdar \& Dhar that gives a
one-to-one map between allowed
configurations and spanning trees \cite{MD92}.
For more details and examples, see also
\cite{Jarai05, mrs3, redig05}.

\emph{Burning algorithm \cite{Dhar}.} Fix
$\eta \in \Omega^{(\gamma)}_\Lambda$, and let
$\xi = \psi_\Lambda(\eta)$. Set
\begin{equation*}\begin{split}
   U_0
  &= \Lambda, \\
   U_1
  &= \{ y \in \Lambda : \xi_y < 2d \}
  = 
  U_0 \setminus \{ y \in U_0=\Lambda : \xi_y \geq 2d \}.
  \end{split}\end{equation*} 
For $t = 1, 2, \dots$, we recursively define
\begin{equation*}
  U_{t+1}
  = U_t \setminus
    \left\{ y \in U_t : \xi_y \ge
    \sum_{z: z\in U_t,\, z \not= y}
    (-\Delta^{(\gamma)}_{zy}) \right\}.
    \end{equation*}
The sites $y$ removed from
$U_t$ to obtain $U_{t+1}$ 
are called ``burnt'' at time $t+1$,  and we say that
they have \textit{burning time} $t(y)=t+1$. 
In particular at time $1$ the sites in $\Lambda$ with height $\geq 2d$ are burnt.
By induction on $t$ and \eqref{fscg},
no site in $\Lambda \setminus U_t$ 
can be contained in any FSC. Hence we have
$\cap_{t=0}^\infty U_t = \varnothing$ 
if and only if
$\xi \in {\mathcal R}^{\mathrm{discr}}_\Lambda$ (if and only if
$\eta \in {\mathcal R}^{(\gamma)}_\Lambda$ by \eqref{e:Requiv}).

Now consider the graph $\widetilde{\mathbb{Z}^d}$ obtained by
adding a new vertex $\varpi$ to $\mathbb{Z}^d$ and
connecting it to every vertex.
Let us call the newly added edges \emph{dissipative}, and
the rest of the edges \emph{ordinary}.
Now we define a new graph $\widetilde{\Lambda}$, by identifying
all vertices in $\mathbb{Z}^d \setminus \Lambda$ with $\varpi$
(and removing loops). In $\widetilde\Lambda$, every
$y \in \Lambda$ is connected to $\varpi$ by exactly one
dissipative edge. Boundary sites of $\Lambda$ are connected
to $\varpi$ by one or more {\em ordinary} edges, in
such a way that $2d$ ordinary edges emanate from each
$y \in \Lambda$. We denote by $E(\widetilde\Lambda)$
the set of edges of $\widetilde\Lambda$.

We define a \textit{spanning tree} ${\mathcal T}_\Lambda(\xi)$ of $\widetilde\Lambda$.
First, for each
$y \in U_0 \setminus U_1$
(that is, when $\xi_y = 2d$), include the dissipative edge of $y$
in the tree.  We define by convention the burning time of $\varpi$ to be 1, $t(\varpi) = 1$.
We already said that sites
$y \in U_0 \setminus U_1$
also have burning time 1: $t(y) = 1$;
and for each $y \in U_1$,
the burning time $t = t(y) \ge 2$ of $y$
is the index $t$ for which
$y \in U_{t-1} \setminus U_t$. 
For $y \in U_{t-1} \setminus U_t$,
i.e., $t(y)=t$, let
\begin{equation*}\begin{split}
 r(y)
  &= | \{ \{ z, y\} \in E(\widetilde\Lambda) :
  \text{$\{ z, y \}$ ordinary and $t(z) = t-1$} \}|, \\
  n(y)
  &= | \{ \{ z, y \} \in E(\widetilde\Lambda) :
  \text{$\{ z, y \}$ ordinary and $t(z) < t$} \}|.
  \end{split}\end{equation*}
 In words, $n(y)$, resp.\ $r(y)$, is the number of neighbors of $y$
excluding $\varpi$ that are burnt before, resp. just before,
$y$ is burnt. 
{}From the construction,  for all $\eta$ such that
$U_1, \dots, U_{t-1}$
take some fixed values, we have the equivalence 
\begin{equation*}\begin{split}
 r(y) = r,\, n(y) = n
  &\qquad \text{if and only if} \qquad
  2d - n \le \eta_y < 2d - n + r \\
  &\qquad \text{if and only if} \qquad
  2d - n \le \xi_y < 2d - n + r.
  \end{split}\end{equation*}
A one-to-one correspondence can be set up between
the $2d$ directions of the ordinary edges and the values
$\{0, 1, \dots, 2d-1\}$. This induces a one-to-one
correspondence between the $r(y)$ ordinary edges and
the values $\{ 2d - n, \dots, 2d - n + r - 1 \}$.
Include in ${\mathcal T}_\Lambda(\xi)$
the ordinary edge $\{ z, y \}$ corresponding to the value
of $\xi_y$. Since each vertex in 
$U_t$ 
is connected to a unique vertex in 
$U_{t-1}$, 
${\mathcal T}_\Lambda(\xi)$ is a spanning tree.
It follows by construction that the mapping
\[
 {\mathcal T}_\Lambda:\xi \mapsto {\mathcal T}_\Lambda(\xi)
\]
is one-to-one and onto the set of spanning trees of $\widetilde\Lambda$.

Let $\mu^{(\gamma)}_\Lambda$ denote the distribution on
$\{ 0, 1 \}^{E(\widetilde\Lambda)}$ under which a spanning tree
$\widetilde t$ has weight
$\gamma^{\widetilde N(\widetilde t)} / \det(\Delta^{(\gamma)}_\Lambda)$,
where $\widetilde N(\widetilde t)$ denotes the number of
dissipative edges in $\widetilde t$.
By construction, for each $\xi$,
$\widetilde N({\mathcal T}_\Lambda (\xi)) = {\mathcal H}(\xi)$ (see \eqref{e:Hxi}),
and therefore, by \eqref{majum},
\[
\mu^{(\gamma)}_\Lambda ({\mathcal T}_\Lambda(\xi))
=\frac{\gamma^{{\mathcal H}(\xi)}}{\det(\Delta^{(\gamma)}_\Lambda)}
=\nu^{(\gamma)}_\Lambda(\xi)= m^{(\gamma)}_\Lambda(\psi_\Lambda^{-1}(\xi)).
\]
\subsection{Waves and spanning trees}
\label{ssec:waves}
We will need an extension of the results of Section \ref{ssec:trees}
that allows us to represent
waves in a $\gamma$-avalanche by spanning trees,
in an analogous way to what happens for waves in an avalanche in the abelian sandpile model
(see \cite{IKP94,IP98,jr}).
Let $\eta \in \Omega^{(\gamma)}_\Lambda$, and suppose we add unit height
at a site, which we assume without loss of generality to be
$0 \in \Lambda$. The $\gamma$-waves created by this addition are defined as
follows. If $\eta_0 + 1 < 2d + \gamma$,
there is no $\gamma$-avalanche and there are no $\gamma$-waves. Assuming
$\eta_0 + 1 \ge 2d + \gamma$,
$\gamma$-topple all sites that can be $\gamma$-toppled, not allowing $0$ to
topple more than once. Then all sites will
topple at most once, and the set of sites that topple,
call it $W^{(1)}_\Lambda(\eta)$, is the first $\gamma$-wave. If after the
first $\gamma$-wave the height at $0$ is still at least $2d + \gamma$,
we start a second $\gamma$-wave, $W^{(2)}_\Lambda$ and so on. We set
$W^{(i)}_\Lambda = \varnothing$, if the $i$-th wave does not exist.
Note that after each $\gamma$-wave, the height at $0$ has decreased by
$\gamma$, and hence the number
of $\gamma$-waves is, deterministically, bounded by
$\lceil \gamma^{-1} \rceil$, where $\lceil a \rceil$ denotes the smallest
integer larger than or equal to a real $a$.

We will represent the intermediate configurations between $\gamma$-waves
as recurrent configurations on an auxiliary space, and
show that they arise by applying the addition operator
at $0$ on this auxiliary space.
Let  for $x,y\in\mathbb{Z}^d$, 
\begin{equation}\label{bobob}\begin{split}
 \widehat{\Omega}^{(\gamma)}_\Lambda
  &= [0, 2d + \gamma + 1) \times
     [0, 2d + \gamma)^{\Lambda \setminus \{ 0 \};} \\
  \widehat{\Delta}^{(\gamma)}_{xy}
  &= \begin{cases}
    2d + \gamma + 1 & \text{if $x=y=0$;} \\
    \Delta^{(\gamma)}_{xy} & \text{otherwise,}
    \end{cases}
    \end{split}\end{equation}
and call
$\widehat{\mathcal{R}}^{(\gamma)}_\Lambda\subset\widehat{\Omega}^{(\gamma)}_\Lambda$
the set of recurrent configurations
for the  toppling matrix  $\widehat{\Delta}^{(\gamma)}_\Lambda$   (which is
the above matrix restricted to $\Lambda$),  and
$\widehat{a}_x$ the corresponding addition operators 
(this amounts to a dissipation $\gamma+1$ at site 0 and $\gamma$ elsewhere).
By Remark \ref{rem:FSC} (which is valid also for
non-homogeneous dissipation, as it is the case here), we have
$\mathcal{R}^{(\gamma)}_\Lambda \subseteq \widehat{\mathcal{R}}^{(\gamma)}_\Lambda$.
Let us show that there is a one-to-one mapping between
$\widehat{\mathcal{R}}^{(\gamma)}_\Lambda \setminus \mathcal{R}^{(\gamma)}_\Lambda$
and intermediate configurations between $\gamma$-waves at $0$.

Let $\eta \in \Omega^{(\gamma)}_\Lambda$. 
Note that if $\eta + \delta_0 \in \mathcal{R}^{(\gamma)}_\Lambda$, then
there are no $\gamma$-waves, and
$\eta + \delta_0 = a_0 \eta = \widehat{a}_0 \eta$.
If $(\eta + \delta_0)_0 = \eta_0 + 1 \ge 2d + \gamma$,
then a $\gamma$-avalanche starts. Put
$\eta^{(1)} :=
\eta + \delta_0 \in \widehat{\mathcal{R}}^{(\gamma)} _\Lambda \setminus \mathcal{R}^{(\gamma)}_\Lambda$.
This is the intermediate configuration before the first $\gamma$-wave, and
we have $\eta^{(1)} = \widehat{a}_0 \eta$.
Carrying out the first $\gamma$-wave in the original model is the same
as applying $\widehat{a}_0$ to $\eta^{(1)}$ in the modified model.
Let $\eta^{(2)} := \widehat{a}_0 \eta^{(1)}$.
If $(\eta^{(2)})_0 < 2d + \gamma$, then there is only one $\gamma$-wave, and
$\eta^{(2)} = a_0 \eta \in \mathcal{R}^{(\gamma)}_\Lambda$ is the configuration
after the $\gamma$-avalanche.
If $(\eta^{(2)})_0 \ge 2d + \gamma$, then $\eta^{(2)}$ is the intermediate
configuration before the second $\gamma$-wave. Performing the
second $\gamma$-wave in the original model amounts to applying
$\widehat{a}_0$ to $\eta^{(2)}$ in the modified model.
The result of the second $\gamma$-wave is $\eta^{(3)} := \widehat{a}_0 \eta^{(2)}$.
We continue inductively until we reach the smallest $K \ge 2$,
such that $\eta^{(K)} \in \mathcal{R}^{(\gamma)}_\Lambda$. This happens
precisely if there were $K-1$ $\gamma$-waves, and then
$\eta^{(K)} = a_0 \eta = \widehat{a}_0^K \eta$. Using invertibility of
$\widehat{a}_0$ on $\widehat{\mathcal{R}}^{(\gamma)}_\Lambda$,
it follows that the intermediate configurations
$\eta^{(1)}, \dots, \eta^{(K-1)}$ are all distinct, and also that
distinct $\eta$'s have distinct intermediate configurations.

We now show that any
$\zeta \in \widehat{\mathcal{R}}^{(\gamma)}_\Lambda \setminus \mathcal{R}^{(\gamma)}_\Lambda$
arises as an intermediate configuration. We first claim that
there exist $n_x \ge 0$ such that
$\prod_x \widehat{a}_x^{n_x} \zeta \in \mathcal{R}^{(\gamma)}_\Lambda$.
To see this, let  $\zeta' := \zeta + \sum_{x \in \Lambda} m_x \delta_x$,  where
$m_x \ge 0$, and $\zeta'_x \ge 2d + \gamma$, $x \in \Lambda$.
The latter condition ensures that
$\mathcal{S}^{(\gamma)}_\Lambda (\zeta') =: \zeta'' \in \mathcal{R}^{(\gamma)}_\Lambda$.
Let $m$ be the number of times $0$ $\gamma$-topples during this stabilization of $\zeta'$.
Then in the modified dynamics, we have
$\widehat{a}_0^m \prod_{x \in \Lambda} \widehat{a}_x^{m_x} \zeta = \zeta''$,
as claimed.
Now define $\zeta''' \in \mathcal{R}^{(\gamma)}_\Lambda$ by the equality
$\zeta'' = \prod_{x \in \Lambda} a_x^{n_x} \zeta'''$. Let $n$ be the number of times
$0$ $\gamma$-topples in computing $\prod_{x \in \Lambda} a_x^{n_x} \zeta'''$. Then we have,
in the modified dynamics,
\begin{equation*}
 \prod_{x \in \Lambda} \widehat{a}_x^{n_x} \zeta
  = \zeta''
  = \widehat{a}_0^n \prod_{x \in \Lambda} \widehat{a}_x^{n_x} \zeta'''.
  \end{equation*}
The above implies $\zeta = \widehat{a}_0^n \zeta'''$. Since
$\zeta''' \in \mathcal{R}^{(\gamma)}_\Lambda$, this proves that
$\zeta$ is an intermediate configuration.

We now extend the spanning tree representation to
$\widehat{\mathcal{R}}^{(\gamma)}_\Lambda$. Put
\begin{equation*}
 \widehat{\Omega}^{\mathrm{discr}}_\Lambda
  = \{ 0, 1, \dots, 2d - 1, 2d, * \} \times
    \{ 0, 1, \dots, 2d - 1, 2d \}^{\Lambda \setminus \{0\}},
    \end{equation*}
and modify $\psi_\Lambda$ by setting
$\widehat{\psi}_\Lambda (\eta)_0 = *$, if
$2d + \gamma \le \eta_0 < 2d + \gamma + 1$.
We define the graph $\widehat{\Lambda}$ by adding an
\emph{extra} edge between $0$ and $\varpi$ in $\widetilde{\Lambda}$.

We use a particular burning rule when applying the
burning bijection to $\widehat{\mathcal{R}}^{(\gamma)}_\Lambda$.
We first send one unit of height along the extra edge
from $\varpi$ to $0$. For configurations in
$\mathcal{R}^{(\gamma)}_\Lambda$, nothing burns. For a configuration
$\eta \in \widehat{\mathcal{R}}^{(\gamma)}_\Lambda \setminus \mathcal{R}^{(\gamma)}_\Lambda$,
a set of sites $W(\eta)$ burns.
This is precisely the $\gamma$-wave in the intermediate
configuration $\eta$. Following this, we send
$\gamma$ units of height along the dissipative edges from $\varpi$,
continue burning, and then send $1$ unit of height along
ordinary edges from $\varpi$, and finish burning.
 Then under the
burning bijection  the spanning trees
containing the extra edge are precisely
the ones representing intermediate configurations, and
the component of $0$ in the forest obtained by removing
the extra edge is the corresponding $\gamma$-wave.

\subsection{Ergodicity  of the finite-volume dynamics}
\label{ssec:ergod}
In this section we show ergodicity of the Markov chain with generator
\eqref{gendis} when $\gamma$ is irrational and consider ergodicity
of single addition operators.
 Because we work with continuous heights, even in finite
volume the continuous-time Markov chain has infinite state space,
and ergodicity is a non-trivial issue.
 The results of Sections \ref{sec:stat-infinite}--\ref{sec:zero-dissip-lim} will not rely on
this section, and therefore the reader can skip this somewhat independent section on first reading. 
\begin{proposition}
\label{prop:ergod}
\begin{itemize}
\item[(a)] For $\gamma$ irrational, the continuous-time Markov chain
with generator \eqref{gendis} is ergodic.
\item[(b)] For $y\in\Lambda$, the addition operator $ a^{(\gamma)}_{\Lambda, y}$
is an ergodic transformation on the measure space
$(\mathcal{R}^{(\gamma)}_\Lambda,m^{(\gamma)}_\Lambda)$
if and only if  $\{ G^{(\gamma)} (x,y): x\in \Lambda\} \cup \{1\}$
is rationally independent.
\end{itemize}
\end{proposition}
\begin{proof}
For  \textit{(a)} we have to show that
$\mathcal{L}_\Lambda f = 0$, $f \in L^2 (m^{(\gamma)}_\Lambda)$
implies $f$ is constant, $m^{(\gamma)}_\Lambda$-a.s.
The set $\mathcal{R}^{(\gamma)}_\Lambda$ is a group under pointwise addition
and $\gamma$-stabilization, isomorphic
to ${\mathbb R}^\Lambda/\Delta^{(\gamma)}_\Lambda$, see \cite{G93}.
Analogously to the discrete case, the characters of this group are indexed
by $m\in \mathbb Z^\Lambda$ via
\begin{equation}\label{carac}
  \chi_m (\eta)
  = \exp \left( 2\pi i\sum_{x,y\in \Lambda}m_x G^{(\gamma)} (x,y) \eta_y\right).
\end{equation}
These characters form a complete orthogonal family in $L^2(m^{(\gamma)}_\Lambda)$.
We have the identity
\begin{equation}\label{stapje}
 \chi_m (a^{(\gamma)}_{\Lambda,z}\eta)
 = \alpha_m (z)\chi_m (\eta),
\end{equation}
where
\begin{equation}\label{theformula}
 \alpha_m (z)
 = \exp\left(2\pi i\sum_{x\in\Lambda} m_x G^{(\gamma)}(x,z)\right).
\end{equation}
The generator applied to $\chi_m$ then gives
\[
 {\mathcal{L}}_\Lambda \chi_m
 = \left(\sum_{x\in\Lambda} {\varphi} (x)\left(\alpha_m(x)-1\right) \right)\chi_m.
\]
So we have to prove that if
\begin{equation}\label{tussenstap}
 \sum_{x\in\Lambda}{\varphi}(x)(\alpha_m(x)-1) =0,
\end{equation}
then $m=0$. Since for all $x\in \Lambda$, $\alpha_m(x)$ is a
complex number of modulus one, and ${\varphi}(x)>0$,
\eqref{tussenstap} implies that $\alpha_m (x)=1$. Hence for all $z\in\Lambda$
\[
 \sum_{x\in\Lambda} m_x G^{(\gamma)}(x,z) = k_z
\]
where $k_z\in \mathbb Z$.
In vector notation this reads  $mG^{(\gamma)}= k$ and gives
$m= k\Delta^{(\gamma)}$, which implies
that for all $x\in \Lambda$,
$k_x (2d+\gamma)$ is an integer. By the irrationality of $\gamma$ this
implies $k_x=0 $, and hence $m= k\Delta^{(\gamma)}=0$.

To prove \textit{(b)}, notice that the ergodicity of $a^{(\gamma)}_{\Lambda, y}$
is equivalent with the statement
that $\chi_m \circ a^{(\gamma)}_{\Lambda, y}= \chi_m$ if and only if
$m=0$. By \eqref{stapje} this is the same as $\alpha_m (y) =1$
if and only if $m=0$.
Now using formula \eqref{theformula} for $\alpha_m$ we have to prove
that
\[
 \sum_{x\in\Lambda} m_x G^{(\gamma)}(x,y)= k
\]
with $k\in \mathbb Z$ implies $m=0$. This is exactly the
condition of linear independence stated.
\end{proof}

The following Proposition gives a more explicit
sufficient condition for ergodicity of an addition
operator. For $x,y \in \Lambda$, let
\begin{equation*}
 P_\Lambda(x,y)
  = \begin{cases}
    \displaystyle{\frac{1}{2d}} & \text{if $x \sim y$;} \\
    0            & \text{otherwise.}
    \end{cases}
    \end{equation*}
Let $\lambda_1 \ge \lambda_2 \ge \dots \ge \lambda_{|\Lambda|}$
be the eigenvalues of $P_\Lambda$, and let
$w_k(x)$, $1 \le k \le |\Lambda|$,
be the corresponding eigenfunctions normalized to have
$\ell_2$-norm $1$.
\begin{proposition}
\label{prop:ergodrect}
\begin{itemize}
\item[(a)] Suppose that $0 \in \Lambda$, the
eigenvalues $\lambda_k$ are all distinct,
and $w_k(0) \not= 0$ for $1 \le k \le |\Lambda|$.
Suppose that $\gamma$ is transcendental.
Then $a^{(\gamma)}_{\Lambda,0}$ is ergodic.
\item[(b)] Examples where the conditions in (a)
are satisfied are given by:
$\Lambda = [-a_1+1,b_1-1] \times \dots \times [-a_d+1,b_d-1]$,
where $p_i = a_i + b_i$ are distinct primes greater than $5$,
$i = 1, \dots, d$.
\end{itemize}
\end{proposition}
\begin{proof}
\textit{(a)} For each $x$,
$\beta_x(\gamma) = G^{(\gamma)}(0,x) = (\Delta^{(\gamma)}_\Lambda )^{-1}_{0x}$
is a ratio of integer polynomials of $\gamma$. We show that under the
assumptions of part \textit{(a)}, the functions
$\{\beta_x : x \in \Lambda \} \cup \{ 1 \}$
are linearly independent over the rationals. This implies that
if $\gamma$ is transcendental, then no rational linear
combination of the numbers $\{ \beta_x (\gamma) : x \in \Lambda \}$
can take a rational value, and hence Proposition \ref{prop:ergod}\textit{(b)}
can be applied.

We have the following spectral representation:
\begin{equation}\begin{split}\label{e:spectral}
 \beta_x(\gamma)
  &= G^{(\gamma)}_\Lambda(0,x)
  = \left[ (2d+\gamma) I - 2d P_\Lambda \right]^{-1} (0,x) \\
  &= \sum_{k=1}^{|\Lambda|} \frac{1}{2d + \gamma - 2d \lambda_k}
    w_k(0) w_k(x),
    \end{split}\end{equation}
     (where $I$ denotes the identity matrix). 
Suppose that $(\alpha(x))_{x \in \Lambda}$ is a system of rationals
such that $\sum_{x \in \Lambda} \alpha(x) \beta_x(\gamma) \equiv c$,
with $c$ a rational constant. Since $\beta_x(\gamma) \to 0$ as
$\gamma \to \infty$, we necessarily have $c = 0$. Hence
\eqref{e:spectral} implies
\begin{equation}\label{e:vanish}
 \sum_{k=1}^{|\Lambda|} \frac{c_k}{\gamma - 2d(\lambda_k - 1)}
  \equiv 0,
  \end{equation}
where $c_k = w_k(0) \sum_{x \in \Lambda} \alpha(x) w_k(x)$.
Since the $\lambda_k$ are all distinct, the sum in
\eqref{e:vanish} can only vanish identically, if
$c_k = 0$ for $1 \le k \le |\Lambda|$. Since $w_k(0) \not= 0$,
this implies that $\sum_{x \in \Lambda} \alpha(x) w_k(x) = 0$
for $1 \le k \le |\Lambda|$. Since the collection of functions
$\{ k \mapsto w_k(x) : x \in \Lambda \}$ forms
a basis, it follows that $\alpha(x) = 0$ for each $x \in \Lambda$.
This proves the stated linear independence, and hence
completes the proof of part \textit{(a)}.

\textit{(b)} For the rectangle $\Lambda$, the eigenfunctions and
eigenvalues can be indexed by
${\bf k} = (k_1, \dots, k_d) \in
\prod_{i=1}^d \{1, \dots, p_i-1 \} =: {\mathcal K}$
and they are given by \cite[Section 8.2]{LL10}:
\begin{equation*}\begin{split}
 \lambda_{\bf k}
  &= \frac1d \sum_{i=1}^d
    \cos \left( \frac{k_i \pi}{p_i} \right); \\
  w_{\bf k} (x)
  &= c_{\bf k} \prod_{i=1}^d
     \sin \left( \frac{(x_i+a_i) k_i \pi}{p_i} \right),
     \end{split}\end{equation*}
where $c_{\bf k}$ is a constant normalizing $w_{\bf k}$ to
have $\ell_2$-norm $1$. Since $p_i$ is prime,
we have $w_{\bf k}(0) \not= 0$ for
${\bf k} \in {\mathcal K}$.

We conclude the proof by showing that when the $p_i$
are as assumed, the eigenvalues are all distinct.
Let $\eta_j = \exp( 2 \pi i k_j / 2 p_j )$,
$\zeta_j = \exp( 2 \pi i l_j / 2 p_j )$, so that
with ${\bf k} = (k_1, \dots, k_d)$ and
${\bf l} = (l_1, \dots, l_d)$ we have
\begin{equation*}
 2 d \lambda_{\bf k} - 2 d \lambda_{\bf l}
  = \sum_{j=1}^d (\eta_j + \eta_j^{-1})
    - \sum_{j=1}^d (\zeta_j + \zeta_j^{-1}).
    \end{equation*}
We prove by induction on $d$ that
$2d (\lambda_{\bf k} - \lambda_{\bf l}) = 0$
implies ${\bf k} = {\bf l}$.

When $d = 1$, we have $\lambda_k = \cos( k \pi / p_1 )$,
$k = 1, \dots, p_1 - 1$, which are all distinct, so the
statement holds for $d = 1$.

For the induction step, we will use some basic results
from field theory, which can be found for example
in \cite[Chapter 8]{Lang05}.
Assume $d \ge 2$, and that
$2d ( \lambda_{\bf k} - \lambda_{\bf l} ) = 0$.
For $m \ge 1$, write $\omega_m = \exp( 2 \pi i / m )$.
Let ${\mathbb Q}(\omega_{m_1}, \dots, \omega_{m_r})$
denote the field obtained by
adjoining $\omega_{m_1}, \dots, \omega_{m_r}$ to the rationals.
The $m$-th cyclotomic
polynomial has roots precisely the primitive
$m$-th roots of unity, is irreducible over ${\mathbb Q}$,
and has degree $ \varphi(m)$, where $ \varphi(m)$ is the
number of residue classes mod $m$
relatively prime to $m$. Hence, since the $p_i$ are
distinct, the degree
of $K := {\mathbb Q}(\omega_{p_1}, \dots, \omega_{p_d})
= {\mathbb Q}(\omega_{p_1 \cdots p_d})$
over ${\mathbb Q}$ is $ \varphi(p_1 \cdots p_d) = (p_1 - 1) \cdots (p_d - 1)$
\cite[Theorem VII.1.3]{Lang05}.
In particular, $\omega_{p_d}$, and in fact all primitive
$p_d$-th roots of unity, are not contained in
$K' := {\mathbb Q}(\omega_{p_1}, \dots, \omega_{p_{d-1}})$, and
they have degree $p_d-1$ over $K'$.
We distinguish three cases.

\emph{Case 1. The numbers $k_d$ and $l_d$ are both even.}
Then $\eta_d = \omega_{p_d}^a$, $\zeta_d = \omega_{p_d}^b$
with $a = k_d/2$, $b = l_d/2$, $1 \le a, b \le (p_d-1)/2$.
Let $\omega$ be a
primitive $p_d$-th root of unity such that
$\omega^{-1} \not= \omega_{p_d}^c$ for $c = a, b, -a, -b$.
Since $p_d > 5$, this is possible.
Since the degree of $\omega$ over $K'$ is $p_d-1$,
the numbers $1, \omega, \omega^2, \dots, \omega^{p_d-2}$
are linearly independent over $K'$. Moreover, they
contain $\eta_d, \eta_d^{-1}, \zeta_d, \zeta_d^{-1}$.
Since $\sum_{j=1}^{d-1} (\eta_j + \eta_j^{-1} - \zeta_j - \zeta_j^{-1}) \in K'$
we need to have $\eta_d = \zeta_d$, and hence
$k_d = l_d$. It follows that
$\sum_{j=1}^{d-1} (\eta_j + \eta_j^{-1} - \zeta_j - \zeta_j^{-1}) = 0$,
and the proof is completed using the induction hypothesis.

\emph{Case 2. The numbers $k_d$ and $l_d$ are both odd.}
Then $\eta_d = - \omega_{p_d}^a$, $\zeta_d = - \omega_{p_d}^b$
with $(p_d+1)/2 \le a, b \le p_d-1$.
Let now $\omega$ be a primitive $p_d$-th root
of unity such that $\omega^{-1} \not= \omega_{p_d}^c$
for $c = a, b, -a, -b$. The argument is then completed
as in Case 1.

\emph{Case 3. The numbers $k_d$ and $l_d$ are of different parity.}
Without loss of generality, assume that $k_d$ is even
and $l_d$ is odd, so that
$\eta_d = \omega_{p_d}^a$, $1 \le a \le (p_d-1)/2$, and
$\zeta_d = - \omega_{p_d}^b$, $(p_d+1)/2 \le b \le p_d-1$.
Again we can find $\omega$ such that
$\omega^{-1} \not= \omega_{p_d}^c$ for $c = a, b, -a, -b$.
Then linear independence implies that this
case is impossible.

The above completes the argument that the eigenvalues
are all distinct, and hence finishes the proof of part \textit{(b)}.
\end{proof}

\section{The infinite-volume limit of the stationary measure of the
abelian avalanche model}
\label{sec:stat-infinite}
\subsection{Existence of the infinite-volume limit of $m^{(\gamma)}_{\Lambda}$}
Recall the definition of $\widetilde\Lambda$
from Section \ref{ssec:trees}.
We can view $\widetilde\Lambda$ as a weighted network, where
ordinary edges have weight $1$, and dissipative edges
have weight $\gamma$. Then $\mu^{(\gamma)}_\Lambda$ is
the weighted spanning tree measure on this network,
that is, the measure where the probability of a tree is
proportional to the product of the weights of the
edges it contains. Wilson's algorithm \cite{Wilson}
can be used to sample from this distribution,
analogously to what is done for the abelian sandpile model, see \cite{aj,jr}.
This is
described as follows. For a path $\sigma$, we denote
by $LE(\sigma)$ its loop-erasure, that is the path
obtained by removing loops from $\sigma$
chronologically. Consider the network random walk
on $\widetilde\Lambda$, that is, the reversible Markov chain
that makes jumps with probabilities proportional to
the weights. Let ${\mathcal F}_0 = \{ \varpi \}$, and let
$x_1, \dots, x_K$ be an enumeration of the vertices
in $\Lambda$. If ${\mathcal F}_{j-1}$ has been defined, start a
network random walk from $x_j$, and run it until it
first hits ${\mathcal F}_{j-1}$. Let $\sigma_j$ be the path
obtained, and let
${\mathcal F}_j = {\mathcal F}_{j-1} \cup LE(\sigma_j)$. By Wilson's
theorem, ${\mathcal F}_K$ is a tree with the stated
distribution.

  Let $\gamma \ge 0$. Using monotonicity arguments, in the spirit of
\cite[Section 5]{benj}, we obtain that for any exhaustion
$\Lambda_1 \subseteq \Lambda_2 \subseteq \dots$ of $\mathbb{Z}^d$, the weak
limit
\begin{equation}\label{muwilson}
\lim_{n \to \infty} \mu^{(\gamma)}_{\Lambda_n}
  =: \mu^{(\gamma)}
\end{equation}
exists.
We will be interested in sampling from
$\mu^{(\gamma)}$ when $\gamma > 0$. For this,
Wilson's algorithm can be used.
When $\gamma = 0$, $d \ge 3$, this is the
statement of \cite[Theorem 5.1]{benj}. The algorithm in
this case uses simple random walk, and the ``root'' is
at infinity. When $\gamma > 0$, the algorithm uses the
network random walk on $\widetilde{\mathbb{Z}^d}$, stopped when it hits
$\varpi$. The tree generated by the algorithm
then gives a sample from $\mu^{(\gamma)}$.
To see the convergence \eqref{muwilson},
couple the algorithms in $\widetilde\Lambda_n$ and
$\widetilde{\mathbb{Z}^d}$, starting walks at a fixed finite
number of sites $x_1, \dots, x_k$. The network random walks
can be coupled until the first time the boundary of $\Lambda_n$
is hit. The coupling shows that the joint distribution
of $(LE(\sigma_1), \dots, LE(\sigma_k))$ in $\widetilde\Lambda_n$
converges as $n \to \infty$ to the joint distribution
in $\widetilde{\mathbb{Z}^d}$. Since this information determines all
finite dimensional distributions of the spanning tree,
the claim \eqref{muwilson} is proven.

Under $\mu^{(\gamma)}$, there is a unique path $\pi_x$
in the tree from $x$ to $\varpi$. Similarly, there is a
unique path $\pi_{\Lambda,x}$ under $\mu^{(\gamma)}_\Lambda$.
Let us write ${\mathcal P}_x$ for the set of self-avoiding paths
from $x$ to $\varpi$ in $\widetilde{\mathbb{Z}^d}$, and
${\mathcal N}_x = \{ y : |y -x| \le 1\}$.
{}From the
correspondence ${\mathcal T}_\Lambda$ we obtain (as in \cite{aj,jr})
\begin{lemma}
\label{lem:height}
Under the map ${\mathcal T}_\Lambda^{-1}$, the height $\xi_x$ is a function
of $\{ \pi_{\Lambda,y} \}_{y \in {\mathcal N}_x}$ only.
\end{lemma}
\begin{lemma}
\label{lem:limit}
For any $\gamma > 0$, and any exhaustion
$\Lambda_1 \subseteq \Lambda_2 \subseteq \dots$ of $\mathbb{Z}^d$, we have
the unique weak limit
\begin{equation*}\begin{split}
{ \lim_{n \to \infty} \nu^{(\gamma)}_{\Lambda_n}
  =: \nu^{(\gamma)}. }
  \end{split}\end{equation*}
Moreover, for all $x\in \mathbb{Z}^d$, there is a map
$h_x : \prod_{y \in {\mathcal N}_x} {\mathcal P}_y \to \{0,1, \dots, 2d\}$,
such that
$\{ \xi_x \}_{x \in \mathbb{Z}^d}$
under $\nu^{(\gamma)}$ has the same law as
$\{ h_x(\pi_y : y \in {\mathcal N}_x) \}_{x \in \mathbb{Z}^d}$ under
$\mu^{(\gamma)}$.
\end{lemma}
Ã¹
\begin{proof}
Coupling Wilson's algorithm in $\widetilde\Lambda_n$ and $\widetilde{\mathbb{Z}^d}$,
we get that the joint distribution of
$(\pi_{\Lambda_n,x_1},\dots, \pi_{\Lambda_n,x_k})$ converges to the
joint distribution of $(\pi_{x_1}, \dots, \pi_{x_k})$.
This and Lemma \ref{lem:height} show that for any
$x_1, \dots, x_k$, the joint distribution under $\nu^{(\gamma)}_{\Lambda_n}$  of
$(\xi_{x_1}, \dots, \xi_{x_k})$
converges to a limit, which has the form stated.
\end{proof}
For the next lemma, we make $m^{(\gamma)}_\Lambda$ into a
measure on $\Omega^{(\gamma)}$ via the
inclusion map
$i : \Omega^{(\gamma)}_\Lambda \to \Omega^{(\gamma)}$,
where
\begin{equation*}
 i(\eta)_x
  = \begin{cases}
    \eta_x & \text{if $x \in \Lambda$;}\\
    0      & \text{if $x \in \mathbb{Z}^d \setminus \Lambda$.}
    \end{cases}
    \end{equation*}

We then obtain the existence of the thermodynamic limit of the
stationary measures $m_\Lambda^{(\gamma)}$ as $\Lambda\uparrow\Zd$.
\begin{theorem}
\label{lem:m-limit}
For any $\gamma > 0$, and any exhaustion
$\Lambda_1 \subseteq \Lambda_2 \subseteq \dots$ of $\mathbb{Z}^d$,
the weak limit
\begin{equation*}\begin{split}
{ \lim_{n \to \infty} m^{(\gamma)}_{\Lambda_n}
  =: m^{(\gamma)} }
  \end{split}\end{equation*}
exists,
and we have $m^{(\gamma)}({\mathcal R}^{(\gamma)}) = 1$.
\end{theorem}
\begin{proof}
Recall the fact that
$m^{(\gamma)}_\Lambda$ is uniform on $(\gamma,\Lambda)$-cells.
Therefore, under the measure $m^{(\gamma)}_\Lambda$ and conditioned
on the value of
$\xi = \psi_\Lambda(\eta) \in \Omega_\Lambda^{\mathrm{discr}}$, the
random variables
$\{ \eta_y - \xi_y \}_{y \in \Lambda}$ are (conditionally)
independent, with distribution
\begin{equation}\begin{split}\label{e:unif}
 \eta_y - \xi_y &\sim \mathrm{Uniform}(0,1)
    \quad\text{for $y$ with $0 \le \xi_y \le 2d-1$;} \\
  \eta_y - \xi_y &\sim \mathrm{Uniform}(0,\gamma)
    \quad\text{for $y$ with $\xi_y = 2d$.}
    \end{split}\end{equation}
Now let $V \subseteq \Lambda$. The above implies that
if we condition on the values of
$\{ \xi_y \}_{y \in V}$ only, then
$\{ \eta_y - \xi_y \}_{y \in V}$ still has the
conditional joint distribution in \eqref{e:unif}
(under $m^{(\gamma)}_\Lambda$).
This and Lemma \ref{lem:limit} imply that
the joint law of $\{ \eta_y \}_{y \in V}$ under
$m^{(\gamma)}_{\Lambda_n}$ converges weakly as
$n \to \infty$, in the space $\Omega^{(\gamma)}_V$.

This proves the weak convergence statement.
Finally, the limit gives
probability $1$ to the event $\{\eta_V\in{\mathcal R}^{(\gamma)}_V\}$ for
any finite $V$, because
\[
m^{(\gamma)}_W \left(\{ \eta: \eta_V\in{\mathcal R}^{(\gamma)}_V\}\right)=1
\]
for all $W\supset V$, since $m^{(\gamma)}_W$
is supported by ${\mathcal R}^{(\gamma)}_W$. 
\end{proof}
\begin{remark}
\label{rem:uniform}
It follows from this proof that under the limiting measure
$m^{(\gamma)}$ we still have $\{ \eta_y - \xi_y \}_{y \in \mathbb{Z}^d}$
conditionally independent, given $\xi = \psi(\eta)$, with
distribution \eqref{e:unif}.  As we will see later on,
when $\gamma = 0$, the distribution
reduces to the first line of \eqref{e:unif}.
\end{remark}
\subsection{Exponential decay of covariances for non-zero dissipation}
We now prove that for dissipation $\gamma>0$, we have
exponential decay of correlations of local observables
in the abelian avalanche model. In \cite{mrs1} this
was shown for large enough dissipation for the  discrete  dissipative sandpile model.

By a \emph{local function}, we mean a function
$f : \Omega^{\mathrm{discr}} \to {\mathbb R}$ (respectively $f: \Omega^{(\gamma)} \to {\mathbb R}$)
such that $f$ depends only
on coordinates in a finite
set $A \subseteq \mathbb{Z}^d$.
\begin{theorem}
\label{ijsbeer}
For all $\gamma > 0$, there exist
$C = C^{(\gamma)}, c > 0$,
such that for all bounded local functions
$f, g : \Omega^{\mathrm{discr}} \to {\mathbb R}$
with dependence sets $A, B$, we have
\begin{equation}\label{e:cov}
 \left| {\mathbb E}_{\nu^{(\gamma)}} [ f g ]
  - {\mathbb E}_{\nu^{(\gamma)}} [ f ]
    {\mathbb E}_{\nu^{(\gamma)}} [ g ] \right|
  \le C |A| |B| \| f \|_\infty \| g \|_\infty
     \exp ( - c\sqrt{\gamma}\, {\mathrm{dist}}(A,B) ).
     \end{equation}
The same statement holds for the measure $m^{(\gamma)}$.
\end{theorem}
\begin{proof}
The first statement deals with discretized heights.
First consider the case when $f$ depends only on $\xi_x$,
and $g$ only on $\xi_y$. Then by Lemma \ref{lem:limit},
$f$ and $g$ depend on the paths $\{ \pi_z \}_{z \in {\mathcal N}_x}$,
and $\{ \pi_w \}_{w \in {\mathcal N}_y}$, respectively.
Use Wilson's algorithm starting with the vertices
in ${\mathcal N}_x$ and then using the vertices in ${\mathcal N}_y$.
We couple the random walks appearing in the algorithm to a new set
of random walks $\{ S_w' \}_{w \in {\mathcal N}_y}$ as follows:
$S_w' = S_w$ until the path hits
$\{ \varpi \} \cup \left(\cup_{z \in {\mathcal N}_x} \pi_z\right)$, and it
moves independently afterwards.
Let $\{ \pi_w' \}_{w \in {\mathcal N}_y}$ be the paths created
by Wilson's algorithm started from ${\mathcal N}_y$,
using the $S_w'$'s. Let $g'$ be a copy of $g$,  that is a
function of $\{\pi_w' : w \in {\mathcal N}_y\}$
instead of $\{\pi_w : w \in {\mathcal N}_y\}$. 
Then the left hand side of \eqref{e:cov} equals
${\mathbb E}_{\nu^{(\gamma)}} [ f (g - g') ]$.  This is bounded by
$2 \| f \|_\infty \| g \|_\infty$ times the probability that
$g \not= g'$. The latter is bounded by the probability
that one of the random walks used for $g$ intersects one of the
random walks used for $f$. This is bounded by
\begin{equation}\label{e:Greenestimate}
 \sum_{z \in {\mathcal N}_x} \sum_{w \in {\mathcal N}_y}
  \sum_{u \in \mathbb{Z}^d} G^{\mathrm{discr}} (z,u)
  G^{\mathrm{discr}} (w,u),
  \end{equation}
where $G^{\mathrm{discr}}$ is the Green's function
 on $\mathbb{Z}^d$  for the discrete
time walk that steps to each neighbor with probability
$1/(2d + \gamma)$ and to $\varpi$ with probability
$\gamma/(2d + \gamma)$. Hence we obtain \eqref{e:cov}
from \eqref{haha}.

In the general case, we can repeat the argument with
the vertices in ${\mathcal M}_f = \cup_{x \in A} {\mathcal N}_x$ and
${\mathcal M}_g = \cup_{y \in B} {\mathcal N}_y$. This leads to an
estimate similar to \eqref{e:Greenestimate}, where
now we sum over $z \in {\mathcal M}_f$ and $w \in {\mathcal M}_g$.
This implies the claim.

Let us prove the statement for the  continuous case,
that is the abelian avalanche model.
Due to the representation \eqref{e:unif}, if we set
\begin{equation*}\begin{split}
 f_0(\xi)
  &= {\mathbb E}_{m^{(\gamma)}} [ f(\eta) |
     \psi(\eta)_{A \cup B} = \xi_{A \cup B} ] \\
  g_0(\xi)
  &= {\mathbb E}_{m^{(\gamma)}} [ g(\eta) |
     \psi(\eta)_{A \cup B} = \xi_{A \cup B} ]
     \end{split}\end{equation*}
then we have
\begin{equation*}\begin{split}
 {\mathbb E}_{m^{(\gamma)}} [ f(\eta) g(\eta) ]
  &= {\mathbb E}_{\nu^{(\gamma)}}[ f_0(\xi) g_0(\xi) ], \\
  {\mathbb E}_{m^{(\gamma)}} [ f(\eta) ]
  &= {\mathbb E}_{\nu^{(\gamma)}} [ f_0(\xi) ], \\
  {\mathbb E}_{m^{(\gamma)}} [ g(\eta) ]
  &= {\mathbb E}_{\nu^{(\gamma)}} [ g_0(\xi) ].
   \end{split}\end{equation*}
Now we can apply the result for the model
with discretized heights to $f_0$ and $g_0$.
\end{proof}
\begin{remark}
Avalanches also satisfy exponential decay when $\gamma > 0$; see
Theorem \ref{thm:waves} item a) below.
\end{remark}
\subsection{Convergence of the stationary measures when $\gamma\to 0$}\label{convergencestat}
For the abelian sandpile model, the
measure $\nu = \nu^{(0)}$ was constructed
in \cite{aj} and \cite[Appendix]{jr}, as the
weak limit of $\nu_\Lambda^{(0)}$.
\begin{theorem}
\label{lem:zerolim}
We have $\lim_{\gamma \to 0} \nu^{(\gamma)} = \nu^{(0)}$,
and $\lim_{\gamma \to 0} m^{(\gamma)} = m^{(0)}$.
\end{theorem}
\begin{proof}
For $\xi_\Lambda\in\Omega^{{\mathrm{discr}}}_\Lambda$ chosen
from $\nu^{(\gamma)}_\Lambda$,
consider the random field of maximal heights
$\{ h_{\Lambda,x} \}_{x \in \Lambda}
:= \{ I[ \xi_{\Lambda,x} = 2d ] \}_{x \in \Lambda}$
(where $I[ \xi_{\Lambda,x} = 2d ] $ denotes the
indicator function of the event $\{ \xi_{\Lambda,x} = 2d \}$), and
the random set
$H_\Lambda = \{ x \in \Lambda : h_{\Lambda,x} = 1 \}$.
Due to the correspondence ${\mathcal T}_\Lambda$, ${\mathbb P} (x \in H_\Lambda)$
equals the probability that the dissipative edge
containing $x$ is included in ${\mathcal T}_\Lambda(\xi_\Lambda)$.
Using Wilson's algorithm started at $x$
and \eqref{majum}, we
see that this probability vanishes in the limit
$\gamma \to 0$, uniformly in $\Lambda$.
Hence
for any finite $V \subseteq \mathbb{Z}^d$,
${\mathbb P}( H_\Lambda \cap V = \varnothing )$ goes to $1$
as $\gamma \to 0$, uniformly in $\Lambda$.

By  \eqref{majum}, after removal of the sites with $\xi_{\Lambda,x}=2d$, the joint
distribution of the heights of the remaining sites is uniform
on allowed configurations with
heights in $\{0,1,\ldots, 2d-1\}$.
Therefore, given $H_\Lambda$,
the conditional distribution of
$\{ \xi_{\Lambda,x} \}_{x \in \Lambda \setminus H_\Lambda}$
is given by $\nu^{(0)}_{\Lambda \setminus H_\Lambda}$,
the
measure for the abelian sandpile model in
$\Lambda \setminus H_\Lambda$. Due to the convergence
$\nu^{(0)}_W \to \nu^{(0)}$, as $W\uparrow\mathbb{Z}^d$,  for large $V$ and
on the event $\{ H_\Lambda \cap V = \varnothing \}$
the conditional distribution is close to
$\nu^{(0)}$, uniformly in $H_\Lambda$ and
$\Lambda \supset V$. The above observations
imply that as $\gamma \to 0$ and $\Lambda \to \mathbb{Z}^d$,
$\nu^{(\gamma)}_\Lambda \to \nu^{(0)}$. This
implies the first statement of the Proposition.

For the second statement, we again use the
representation \eqref{e:unif}. Then the
convergence of finite-dimensional distributions
follows from the first part of the Proposition.
Tightness holds trivially since all the measures
under consideration have support contained
in (the same) compact set.
\end{proof}
\begin{remark}
\label{rem:determ}
The process
$\{ h_x \}_{x \in \mathbb{Z}^d}$ is a determinantal
process \cite{HKPV06},
that is there exists a kernel ${\mathbf K}^{(\gamma)}(x,y)$, such that
for $n \ge 1$ and distinct $x_1, \dots, x_n \in \mathbb{Z}^d$,
\begin{equation*}
 {\mathbb P} (h_{x_1} = 1,\dots,h_{x_n} = 1)
  = \det ({\mathbf K}^{(\gamma)}(x_i,x_j))_{i,j=1}^n.
  \end{equation*}
This follows from the Transfer Current Theorem
of Burton \& Pemantle \cite{bupe} applied to
the collection of dissipative edges.
\end{remark}
\subsection{Speed of convergence for special events}
The following proposition gives an estimate of the speed of convergence
in Theorem \ref{lem:zerolim} for the probability of having
mimimal height, for which the method of Dhar \& Majumdar
(see \cite{md2,mru}) can be applied. Power law estimates for
general cylinder events are proved in \cite{Jarai10}.

We call a set of points $\{x_1,\ldots,x_k\}\subseteq \mathbb{Z}^d$ isolated from
each other if there does not exist $y\in\mathbb{Z}^d$ that is a neighbor
of more than one site in the set. Denote
\[
  E_{x_1,\ldots,x_k}
  = \left\{\eta\in{\mathcal X}:\forall i\in \{ 1,\ldots,k\},0\leq \eta_{x_i} <1\right\}
\]
\begin{proposition}
\label{prop:rate}
For $d \ge 2$ and $k\in{\mathbb N}$ there exists $C(k,d) >0$ such that
for all sets $\{x_1,\ldots, x_k\}$ of isolated points we have
\begin{equation}
|m^{(\gamma)} (E_{x_1,\ldots,x_k})
-m^{(0)} (E_{x_1,\ldots,x_k})|
\leq \begin{cases}
  C(k,d) \gamma & \text{if $d \ge 3$;}\\
  C(k,2) \gamma \log(1/\gamma) & \text{if $d = 2$.}
     \end{cases}
\end{equation}
\end{proposition}
\begin{proof}
\textit{(a)} Start with $d\geq 3$.
First remark that, using
the notation of section \ref{ssec:trees}, for $\gamma\geq 0$,
$m^{(\gamma)}(E_{x_1,\ldots,x_k})= \nu^{(\gamma)}(\xi_{x_1}=0,\ldots,\xi_{x_k}=0)$.
In terms of the discretized height-configuration, this probability
can be computed using the method of \cite{md2}: the allowed
discretized configurations with $(\xi_{x_1}=0,\ldots,\xi_{x_k}=0)$
are in one-to-one correspondence with allowed configurations
on a modified graph. The latter is obtained by removing
the dissipative edge at each $x_j$ and all ordinary edges emanating from each
$x_j$, except the one leading to $x_j + (1,0,\dots,0)$.
Then
\[
\nu^{(\gamma)} (\xi_{x_1}=0,\ldots,\xi_{x_k}=0)= \det (I + K^{(\gamma)} B^{(\gamma)}),
\]
where $B^{(\gamma)}$, $K^{(\gamma)}$ are finite matrices
(and $I$ denotes the identity matrix)
with $B^{(\gamma)}$ not depending on the particular
sites in $\{x_1,\ldots x_k\}$, and elements of $K^{(\gamma)}$
have the form $G^{(\gamma)}(u,v)$. Moreover the off-diagonal elements
of $B^{(\gamma)}$ do not depend on $\gamma$ and for the diagonal
elements we have $|B^{(\gamma)}(x,x)-B^{(0)} (x,x)| = \gamma$.
We have the uniform estimate
\[
 0\leq G^{(0)}(x,y)-G^{(\gamma)} (x,y)\leq C \gamma.
\]
This can be seen for example from the Fourier
representation   (see \cite[Propositions 4.2.3, 4.4.3]{LL10}, \cite{spitzer})
\[
 G^{(\gamma)} (x,y) = \int_{[0,2\pi)^d} \frac{e^{ik(x-y)}}{\Gamma^{(\gamma)} (k)} dk
\]
with
\[
 \Gamma^{(\gamma)} (k)
 = 1- (1-\gamma) \left(\frac1{2d}\sum_{e\in\mathbb{Z}^d:|e|=1} \cos (k\cdot e) \right).
\]
This gives
\[
  | \mbox{det} (I + G^{(\gamma)} B^{(\gamma)}) - \mbox{det} (I + G^{(0)} B^{(0)})|
  \leq C(k) \gamma.
\]
\textit{(b)} In the case $d=2$ we cannot directly work in infinite volume,
as the Green's function is divergent.
So we proceed by finite volume approximations: take $\Lambda $ big enough
so that it contains $x_1,\ldots, x_k$ and all their neighbors:
\[
 \nu^{(\gamma)}_\Lambda (\xi_{x_1}=0,\ldots,\xi_{x_k}=0)
 = \det (I + K^{(\gamma)}_\Lambda B^{(\gamma)}).
\]
The matrix $B^{(0)}$ has zero column sums.
Therefore, $A^{(\gamma)}_\Lambda B^{(0)}=0$ for a matrix $A^{(\gamma)}_\Lambda$
with identical entries equal to $G_\Lambda^{(\gamma)} (0,0)$.
Hence, we can write
\[
 K^{(\gamma)}_\Lambda B^{(\gamma)}-K_\Lambda^{(0)} B^{(0)}
 = K_\Lambda^{(\gamma)} (B^{(\gamma)}-B^{(0)})
   + \left((K^{(\gamma)}_\Lambda- A^{(\gamma)}_\Lambda)
   - (K^{(0)}_\Lambda-A^{(0)}_\Lambda)\right) B^{(0)}.
\]
For a matrix $D^{(\gamma)}$, and a function
$f:(0,1]\to (0,\infty)$ we write $D^{(\gamma)}= O(f(\gamma))$
if for all entries $(i,j)$,
$|D^{(\gamma)}_{ij} |\leq C f(\gamma)$.
Then if $f$ is such that $f(\gamma)\to 0$ as $\gamma\to 0$,
and if $D-D'= O(f(\gamma))$ we have
$|\det (I +D)- \det (I+D') |\leq C' f(\gamma)$.

Therefore, in order to prove the statement for $d=2$,
if suffices to see that
\[
 K^{(\gamma)}_\Lambda B^{(\gamma)}-K_\Lambda^{(0)} B^{(0)}= O(\gamma\log(1/\gamma)).
\]
By Fourier representation of the potential kernel in $d=2$
\[
 (K^{(\gamma)}_\Lambda- A^{(\gamma)}_\Lambda)- (K^{(0)}_\Lambda-A^{(0)}_\Lambda)= O(\gamma)
\]
uniformly in $\Lambda$. Furthermore, uniformly in $\Lambda$,
\[
 K_\Lambda^{(\gamma)} (B^{(\gamma)}-B^{(0)})= O(\gamma \log (1/\gamma) ).
\]
\end{proof}
\section{The dynamics of the dissipative model}
\label{sec:dynamics}
\subsection{Legal and exhaustive sequences of topplings}
Let ${\varphi} : \mathbb{Z}^d \to (0, \infty)$ be a bounded function, and
let $(N^{\varphi}_t)_{t\geq 0}:= \{ N^{\varphi}_{x,t} \}_{x \in \mathbb{Z}^d,\, t \ge 0}$ be a
collection of independent Poisson processes
on a
probability space $(X, {\mathcal F}, {\mathbb P})$, where
$(N^{\varphi}_{x,t})_{t\geq 0}$ has rate ${\varphi}(x)$.
We want to define the
dynamics $\{ \eta_t \}_{t \ge 0}$ of the dissipative model
for an initial configuration $\eta_0$ as the result of
stabilizing $\eta_0 + N^{\varphi}_t$.  Since this typically
involves infinitely many topplings, it requires some care.
The result of the next two lemmas will be that the
procedure is well-defined.

First we consider $\gamma$-stabilization of infinite configurations.
Fix $\eta \in {\mathcal X}$, and a sequence $x_1, x_2, \dots$ of
legal $\gamma$-topplings.
\begin{definition}\label{exhau}
A sequence $x_1, x_2, \ldots$ of
legal $\gamma$-topplings
of a configuration $\eta$ is called $\gamma$-\emph{exhaustive},
if for every $n\in{\mathbb N}$,
and for every $\gamma$-unstable site $x$ of
$T_{x_n}\circ\ldots\circ T_{x_1} (\eta)$,
there exists $m>n$ such that $x_m=x$.
A sequence $x_1, x_2, \ldots$ of
legal $\gamma$-topplings
of a configuration $\eta$ is called $\gamma$-\emph{stabilizing}
if the limit $\lim_{n\to\infty}T_{x_n}\circ\ldots\circ T_{x_1} (\eta)$
is $\gamma$-stable.
\end{definition}
The notion of an exhaustive sequence of topplings slightly generalizes what
in the finite case corresponds to a stabilizing sequence of topplings. In particular
in an exhaustive sequence sites are allowed to topple infinitely many times.
Just as in the finite case, the number of times each site
topples in a $\gamma$-exhaustive sequence (which may be infinite)
is independent of the sequence (see \cite[Lemma 4.1]{JL07} for a proof).
It can also be seen the same
way that if $y_1, y_2, \dots$ is another legal
sequence of $\gamma$-topplings, then each site $\gamma$-topples at most as many
times as in a $\gamma$-exhaustive sequence.
Call
\begin{equation*}
 (N^{(\gamma)}(\eta))_x
  = \text{the number of times $x$ $\gamma$-topples
    in a $\gamma$-exhaustive sequence}.
    \end{equation*}
We say that $\eta \in {\mathcal X}$ is \emph{$\gamma$-stabilizable}, if
$(N^{(\gamma)}(\eta))_x < \infty$ for all $x \in \mathbb{Z}^d$. In this
case, similarly to \eqref{e:stabilize}, the $\gamma$-stabilization is
related to the original configuration by the formula:
\begin{equation*}
 {\mathcal S}^{(\gamma)}(\eta)
  = \eta - \Delta^{(\gamma)}N^{(\gamma)}(\eta) .
  \end{equation*}
Note that every $\gamma$-stabilizing sequence is $\gamma$-exhaustive and if the configuration
is $\gamma$-stabilizable, every $\gamma$-exhaustive sequence is $\gamma$-stabilizing.
Recall that $(N^{(\gamma)}_\Lambda(\eta))_x$ denotes the number
of times $x \in \Lambda$ $\gamma$-topples during $\gamma$-stabilization in $\Lambda$.
We will see in Lemma \ref{lem:dynamics} that any
configuration that does not grow too fast at infinity
is $\gamma$-stabilizable.
\begin{lemma}
\label{lem:stabilize}
Let $\gamma \ge 0$, $\eta\in {\mathcal X}$. \\
(i) The vector $N^{(\gamma)}_\Lambda(\eta)$ is componentwise
monotone increasing in $\Lambda$ and $\eta$.\\
(ii) If $\eta$ is $\gamma$-stabilizable,
we have $(N^{(\gamma)}(\eta))_x = \sup_\Lambda (N^{(\gamma)}_\Lambda(\eta))_x$,
$x \in \mathbb{Z}^d$.\\
(iii) If $\eta$ is $\gamma$-stabilizable, we have
${\mathcal S}^{(\gamma)}(\eta) = \lim_\Lambda {\mathcal S}^{(\gamma)}_\Lambda(\eta)
 \in \Omega^{(\gamma)}$.\\
(iv) If $\zeta_1, \zeta_2 \in {\mathcal X}$, and $\zeta_1 + \zeta_2$ is
$\gamma$-stabilizable, then we have
\begin{equation}\label{e:Abelianstab}
 {\mathcal S}^{(\gamma)}( \zeta_1 + \zeta_2 )
  = {\mathcal S}^{(\gamma)}( {\mathcal S}^{(\gamma)}(\zeta_1) + \zeta_2 ).
  \end{equation}
\end{lemma}
\begin{proof}
{\em (i)} Consider $\Lambda \subseteq \Lambda'$.   First $\gamma$-stabilize $\eta$ in $\Lambda$,
and record the amount of height flowing out of $\Lambda$.
Now $\gamma$-stabilize in $\Lambda'$, with further $\gamma$-topplings,
if necessary. To prove the other statement, consider
$\eta \le \eta'$. First $\gamma$-stabilize $\eta$. Adding height
$\eta' - \eta$ does not affect the
legality of the
sequence of $\gamma$-topplings. Hence we can $\gamma$-stabilize $\eta'$
by further $\gamma$-topplings if necessary.

\noindent{\em (ii)} Fix a sequence $\Lambda_1 \subseteq \Lambda_2 \subseteq \dots$ such
that $\cup_k \Lambda_k = \mathbb{Z}^d$. First $\gamma$-stabilize $\eta$ in $\Lambda_1$, then
in $\Lambda_2$, and so on. We thus get a $\gamma$-exhaustive sequence, and
each site $x$ $\gamma$-topples $\sup_\Lambda (N^{(\gamma)}_\Lambda\eta)_x$ times.

\noindent{\em (iii)} Note that if $\eta$ is $\gamma$-stabilizable, we have
$N^{(\gamma)}\eta = \lim_\Lambda (N^{(\gamma)}_\Lambda\eta)$, and hence
by \eqref{e:stabilize}
\begin{equation*}
 {\mathcal S}^{(\gamma)}(\eta)
  = \eta - \Delta^{(\gamma)} N^{(\gamma)}\eta
  = \lim_\Lambda \left[ \eta - \Delta^{(\gamma)} N^{(\gamma)}_\Lambda\eta \right]
  = \lim_{\Lambda} {\mathcal S}^{(\gamma)}_\Lambda (\eta).
  \end{equation*}
\noindent{\em (iv)}  This is essentially not more than the
Abelian property. However, since there can be infinitely many
$\gamma$-topplings in $\gamma$-stabilizing $\zeta_1$, care needs to be taken.
Note that $\zeta_1 \le \zeta_1 + \zeta_2$
implies that any
legal sequence for
$\zeta_1$ is
legal for $\zeta_1 + \zeta_2$ as well.
In particular, $\zeta_1$ is also $\gamma$-stabilizable.
Fix a $\gamma$-stabilizing sequence $y_1, y_2, \dots$ for $\zeta_1$.
We now construct a $\gamma$-exhaustive sequence for
$\zeta_1 + \zeta_2$. Fix again
$\Lambda_1 \subseteq \Lambda_2 \subseteq \dots$
We write
\begin{equation*}\begin{split}
 \partial_+ \Lambda
  &= \{ x \in \Lambda^c :
     \text{$\exists\, y \in \Lambda$ such that $x \sim y$} \}; \\
  \partial_- \Lambda
  &= \{ y \in \Lambda :
     \text{$\exists\, x \in \Lambda^c$ such that $x \sim y$} \}; \\
     \overline\Lambda
  &= \Lambda \cup \partial_+ \Lambda.
  \end{split}\end{equation*}
Let $m_1$ be an index such that, for all
$y \in \overline{\Lambda_1}$,
\begin{equation*}
 | \{ 1 \le k \le m_1 : y_k = y \} |
  = (N^{(\gamma)}\zeta_1)_y.
  \end{equation*}
i.e., after index $m_1$ there will be no further $\gamma$-topplings
in $\overline{\Lambda_1}$ in the $\gamma$-stabilization of $\zeta_1$.
Note that $y_1, \dots, y_{m_1}$ is
legal for
$\zeta_1 + \zeta_2$. Now stabilize
$T_{y_{m_1}} \circ \dots T_{y_1} (\zeta_1 + \zeta_2)$
in $\Lambda_1 \setminus \partial_- \Lambda_1$, recording the amount
of height flowing out of $\Lambda_1 \setminus \partial_- \Lambda_1$.
This leads to potential further $\gamma$-topplings in
$\Lambda_1 \setminus \partial_- \Lambda_1$ at
$z_1, \dots, z_{n_1}$. By construction,
$y_1, \dots, y_{m_1}, z_1, \dots, z_{n_1}$ is
legal for
$\zeta_1 + \zeta_2$, and $z_1, \dots, z_{n_1}$ is
legal for
${\mathcal S}^{(\gamma)}(\zeta_1) + \zeta_2$. Moreover, the configuration in
$\Lambda_1 \setminus \partial_- \Lambda_1$ after $\gamma$-topplings at
$y_1, \dots, y_{m_1}, z_1, \dots, z_{n_1}$ coincides with
the $\gamma$-stabilization of ${\mathcal S}^{(\gamma)}(\zeta_1) + \zeta_2$ in
$\Lambda_1 \setminus \partial_- \Lambda_1$.

Now select an index $m_2 \ge m_1$, such that for all
$y \in \overline{\Lambda_2}$,
\begin{equation*}
 | \{ 1 \le k \le m_2 : y_k = y \} |
  = (N^{(\gamma)}\zeta_1)_y.
  \end{equation*}
Since the $\gamma$-topplings at $z_1, \dots, z_{n_1}$ do not
change the configuration in $\Lambda_1^c$, the sequence
\begin{equation*}
 y_1, \dots, y_{m_1},
  z_1, \dots, z_{n_1},
  y_{m_1+1}, \dots, y_{m_2}
  \end{equation*}
is legal for $\zeta_1 + \zeta_2$. Now we $\gamma$-stabilize in
$\Lambda_2 \setminus \partial_- \Lambda_2$ via $\gamma$-topplings
at $z_{n_1+1}, \dots, z_{n_2}$. Again by construction,
\begin{equation*}
 y_1, \dots, y_{m_1},
  z_1, \dots, z_{n_1},
  y_{m_1+1}, \dots, y_{m_2},
  z_{n_1+1}, \dots, z_{n_2}
  \end{equation*}
is legal for $\zeta_1 + \zeta_2$, and
\begin{equation*}
 z_1, \dots, z_{n_1},
  z_{n_1+1}, \dots, z_{n_2}
  \end{equation*}
is legal for ${\mathcal S}^{(\gamma)}(\zeta_1) + \zeta_2$.
Continuing this argument, we obtain an interlacement
of two sequences $y_1, y_2, \dots$ and $z_1, z_2, \dots$
$\gamma$-exhaustive for
$\zeta_1 + \zeta_2$, and hence its final result is
${\mathcal S}^{(\gamma)}(\zeta_1 + \zeta_2)$. On the other hand, the
configuration in $\Lambda_k \setminus \partial_- \Lambda_k$
after the $\gamma$-topplings at
\begin{equation*}
y_1, \dots, y_{m_1},
  z_1, \dots, z_{n_1},
  \dots
  y_{m_{k-1}+1}, \dots, y_{m_k},
  z_{n_{k-1}+1}, \dots, z_{n_k}
  \end{equation*}
coincides with the $\gamma$-stabilization of
${\mathcal S}^{(\gamma)}(\zeta_1) + \zeta_2$ in
$\Lambda_k \setminus \partial_- \Lambda_k$.
This yields the claim.
\end{proof}
\begin{lemma}
\label{lem:dynamics}
Let $\gamma > 0$. \\
(i) If $\eta \in {\mathcal X}$ satisfies
$\sum_{y \in \mathbb{Z}^d} G^{(\gamma)}(x,y) \eta_y < \infty$ for
all $x \in \mathbb{Z}^d$, then $\eta$ is $\gamma$-stabilizable.\\
(ii) With ${\mathbb P}$-probability $1$, for any $\eta_0 \in \Omega^{(\gamma)}$
and any $t \ge 0$, $\eta_0 + N^{\varphi}_t$ is $\gamma$-stabilizable.\\
\end{lemma}
\begin{proof}
{\em (i)} By \eqref{e:stabilize},
\begin{equation*}
\Delta^{(\gamma)}_\Lambda N^{(\gamma)}_\Lambda\eta
  = \eta - {\mathcal S}^{(\gamma)}_\Lambda (\eta)
  \le \eta.
  \end{equation*}
Hence, by \eqref{haha}
\begin{equation}\begin{split}\label{e:numtop}
 \sup_\Lambda (N^{(\gamma)}_\Lambda\eta)_x
  \le \sup_\Lambda \sum_{y \in \Lambda} G^{(\gamma)}_\Lambda(x,y)
    \eta_y
  \le \sum_{y \in \mathbb{Z}^d} G^{(\gamma)}(x,y)
    \eta_y
  < \infty.
  \end{split}\end{equation}
{\em (ii)} Due to the boundedness of ${\varphi}$ and estimate
\eqref{haha}, the configuration
$\eta_0 + N^{\varphi}_t$ satisfies the condition in part {\em (i)}
for all $t \ge 0$, with probability 1.
\end{proof}
By Lemmas \ref{lem:stabilize} and \ref{lem:dynamics},
the process
\begin{equation}\label{e:processdef}
 \eta_t
  = {\mathcal S}^{(\gamma)}( \eta_0 + N^{\varphi}_t )
  \end{equation}
is well-defined for any initial configuration
$\eta_0 \in \Omega^{(\gamma)}$.
The computation in \eqref{e:numtop} also gives that the
addition operators
\begin{equation*}
 a^{(\gamma)}_x \eta
  := {\mathcal S}^{(\gamma)} (\eta + \delta_x)
  = \lim_\Lambda {\mathcal S}^{(\gamma)}_\Lambda (\eta + \delta_x)
  \end{equation*}
 are defined for any $\eta \in \Omega^{(\gamma)}$, $\gamma > 0$.

\subsection{Finiteness of avalanches}
Recall that $n^{(\gamma)}_\Lambda (x,y,\eta_0)$
denotes the number of $\gamma$-topplings
occurring at $y$ in computing
${\mathcal S}^{(\gamma)}_\Lambda(\eta_0 + \delta_x)$.
We also define $n^{(\gamma)}(x,y,\eta_0)
:= \sup_\Lambda n^{(\gamma)}_\Lambda(x,y,\eta_0)$.
We call the sets 
\begin{equation}\begin{split}\label{finvolav}
 {\mathrm{Av}}^{(\gamma)}_{\Lambda,x}(\eta)
  &= \{ y \in \Lambda : n^{(\gamma)}_\Lambda(x,y,\eta) \geq 1 \} \\
  {\mathrm{Av}}^{(\gamma)}_x(\eta)
  &= \{ y \in \mathbb{Z}^d : n^{(\gamma)}(x,y,\eta) \geq 1 \}.
\end{split}\end{equation}
the \emph{$\gamma$-avalanche clusters}
started at $x$ in $\Lambda$, resp.~in $\mathbb{Z}^d$. 
\begin{lemma}
\label{lem:invariance}
Let $\gamma > 0$.\\
(i) We have
$m^{(\gamma)} \left( |{\mathrm{Av}}^{(\gamma)}_x(\eta)| < \infty \right)
= 1$.\\
(ii) The transformation $a^{(\gamma)}_x$ leaves $m^{(\gamma)}$
invariant.
\end{lemma}
\begin{proof}
Take $\eta_0$ distributed according to $m^{(\gamma)}$.
Due to \eqref{dhar},
\begin{equation}\begin{split}\label{e:fromdhar}
 {\mathbb E}_{m^{(\gamma)}} [ n^{(\gamma)}(x,y,\eta_0) ]
  &= \lim_\Lambda
    {\mathbb E}_{m^{(\gamma)}} [ n^{(\gamma)}_\Lambda(x,y,\eta_0) ] \\
  &= \lim_\Lambda \lim_V
    {\mathbb E}_{m^{(\gamma)}_V} [ n^{(\gamma)}_\Lambda(x,y,\eta_0) ] \\
  &\le \lim_V
    {\mathbb E}_{m^{(\gamma)}_V} [ n^{(\gamma)}_V (x,y,\eta_0) ] \\
  &= G^{(\gamma)}(x,y).
  \end{split}\end{equation}
Inequality \eqref{e:fromdhar} and the estimates \eqref{markest}, \eqref{haha}
yield ${\mathbb E}_{m^{(\gamma)}} |{\mathrm{Av}}^{(\gamma)}_x(\eta)| < \infty$,
for any $x$, implying {\em (i)}. As in
\cite[Section 4]{jr}, we have that {\em (i)} implies {\em (ii)}.
\end{proof}
\subsection{The stationary Markov process with positive dissipation and bounded addition rate}
\begin{theorem}
\label{thm:dissip-dyn}
Let $\gamma > 0$. \\
(i) The process $\{ \eta_t \}_{t \ge 0}$ is Markovian.\\
(ii) With ${\mathbb P}$-probability $1$, for any
$\eta_0 \in \Omega^{(\gamma)}$ we have
\begin{equation*}
\eta_t
  = \lim_V \prod_{x \in V}
    \left( a^{(\gamma)}_x \right)^{N^{\varphi}_{x,t}} (\eta_0).
    \end{equation*}
(iii) If $\eta_0$ is distributed according to $m^{(\gamma)}$
then $\{ \eta_t \}_{t \ge 0}$ is stationary.
\end{theorem}
\begin{proof}
{\em (i)}
By \eqref{e:processdef}, for $0\le t\le t+s$, and $\eta_0\in \Omega^{(\gamma)}$,
$\eta_t$ and $\eta_{t+s}$ are well-defined, and
$\eta_0 + N^{\varphi}_t$, $\eta_0 + N_{t+s}^{\varphi}$ are
$\gamma$-stabilizable
by Lemma \ref{lem:dynamics} {\em (ii)}.
By Lemma \ref{lem:stabilize} {\em (iv)}, we have
\begin{equation*}
 \eta_{t+s}
  = {\mathcal S}^{(\gamma)}(\eta_0 + N^{\varphi}_{t+s})
  = {\mathcal S}^{(\gamma)}( \eta_t + N^{\varphi}_{t+s} - N^{\varphi}_t ),
    \qquad s \ge 0.
    \end{equation*}
This implies the Markov property.

{\em (ii)} Condition on a realization of the Poisson processes
such that, for all $y\in\mathbb{Z}^d$,
$\sum_{x \in \mathbb{Z}^d} N^{\varphi}_{x,t} G^{(\gamma)}(x,y) < \infty$.
Let $x_1, x_2, \dots$ be an enumeration of $\mathbb{Z}^d$, and let
\begin{equation*}\begin{split}
 \eta_n
  &= \eta_0 + \sum_{i=1}^n N^{\varphi}_{x_i,t} \delta_{x_i}; \\
  \eta
  &= \eta_0 + N^{\varphi}_t; \\
  \zeta_n
  &= {\mathcal S}^{(\gamma)}(\eta_n)
  = \prod_{i=1}^n \left(
    a^{(\gamma)}_{x_i} \right)^{N^{\varphi}_{x_i,t}} (\eta_0)\\
  \zeta
  &= {\mathcal S}^{(\gamma)}(\eta).
  \end{split}\end{equation*}
Note that $\zeta$ is indeed well-defined
by Lemma \ref{lem:dynamics}{\em (ii)}.
Given $W \subseteq \mathbb{Z}^d$ finite, select $\Lambda$ such that
$(N^{(\gamma)}_\Lambda\eta)_y = (N^{(\gamma)}\eta)_y$ for all
$y \in \overline W$.
If $n \ge n_0(\Lambda)$, we have $\eta_n = \eta$ in $\Lambda$,
and therefore
\begin{equation*}
 (N^{(\gamma)}(\eta))_y
  = (N_\Lambda^{(\gamma)}(\eta))_y
  = (N_\Lambda^{(\gamma)}(\eta_n))_y
  \le (N^{(\gamma)}(\eta_n))_y
  \le (N^{(\gamma)}(\eta))_y, \qquad y \in\overline W,
  \end{equation*}
where the first inequality is due to
Lemma \ref{lem:stabilize}{\em (ii)}, and the second to
Lemma \ref{lem:stabilize}{\em (i)}. This shows that
$(\zeta_n)_y = \zeta_y$ for $y \in W$, when
$n \ge n_0(\Lambda)$. Hence $\zeta_n \to \zeta$, proving {\em (ii)}.

{\em (iii)} Due to Lemma \ref{lem:invariance}{\em (ii)}, and part {\em (ii)},
$\eta_t$ is an almost sure limit of configurations
distributed according to $m^{(\gamma)}$. This completes
the proof.
\end{proof}
\section{Avalanche tails}\label{avtail}
\subsection{Upper and lower bounds on avalanche tails}
The next theorem gives upper and lower bounds on the
probability that a $\gamma$-avalanche started at $0$ contains a
vertex $x$ in the infinite-volume system. The upper bound
is essentially Dhar's formula.
In Theorem \ref{thm:new-av-lb} we will improve the lower bound in part (a) to
one that matches the upper bound up to a multiplicative constant when $d > 4$.
Part (b), that follows by essentially the same proof as part (a), concerns the
following modified model: let the dissipation at the origin be $0 \le \gamma_0 < 1$,
but keep the dissipation at all other vertices $\gamma$. We refer to this
model by superscripts $(\gamma_0,\gamma)$ in our notation.
\begin{theorem}
\label{thm:waves}
Let $d \ge 1$ and $0 \le \gamma < 1$.
\begin{itemize}
\item[(a)] We have
\begin{equation}\label{avesti}
 \gamma G^{(\gamma)} (0,x)
 \leq m^{(\gamma)} (x\in {\mathrm{Av}}_0^{(\gamma)} (\eta))
 \leq G^{(\gamma)} (0,x).
\end{equation}
\item[(b)] There are constants $0 < c = c(d) < C = C(d)$, such that
\begin{equation}\label{doublegamma}
 c\, \gamma_0\, G^{(\gamma)} (0,x)
 \le m^{(\gamma_0,\gamma)} (x \in \Av_0^{(\gamma_0,\gamma)} (\eta))
 \le C\, G^{(\gamma)} (0,x).
\end{equation}
\end{itemize}
\end{theorem}
\begin{proof}
The upper bound in \textit{(a)} is immediate from Dhar's formula
and Markov's inequality. 
We next prove the lower bound in \textit{(a)}.
We first work in finite volume $\Lambda$. Recall the representation of $\gamma$-waves
from Section \ref{ssec:waves}. For any $\zeta\in \widehat{\mathcal{R}}^{(\gamma)}_\Lambda\setminus\mathcal{R}^{(\gamma)}_\Lambda$
we denote by $W(\zeta)$ the set of those sites that have to be toppled stabilizing
$\zeta+\delta_0$ in order to obtain $\widehat{a}_0\zeta$, i.e., using the toppling matrix
$\widehat{\Delta}^{(\gamma)}$ from \eqref{bobob}.
Let $\widehat{B}_x \subseteq \widehat{\mathcal{R}}^{(\gamma)}_\Lambda$ denote
the set of those intermediate configurations whose wave contains $x$.
Then $\widehat{m}^{(\gamma)}_\Lambda(\widehat{B}_x)$ is the probability
that the weighted random spanning tree in $\widehat{\Lambda}$
contains the extra edge, and the component of $0$ in the forest
obtained by removing the extra edge contains $x$. By Wilson's
algorithm, this is the same as the probability that a
random walk started at $x$ first reaches $\varpi$ via the
extra edge, which can be computed as:
\begin{equation}\label{e:wavesprob}
 \widehat{m}^{(\gamma)}_\Lambda(\widehat{B}_x)
  = \frac{G^{(\gamma)}_\Lambda(0,x)}{1 + G^{(\gamma)}_\Lambda(0,0)}.
  \end{equation}
Indeed, $\widehat{m}^{(\gamma)}_\Lambda(\widehat{B}_x)$ is the probability that the random
walk hits zero starting from $x$ which is
equal to $\frac{G^{(\gamma)}_\Lambda(0,x)}{G^{(\gamma)}_\Lambda(0,0)}$ times
the probability that the random walk starting from zero reaches
$\varpi$ via the extra edge. The latter probability equals
$\widehat{m}^{(\gamma)}_\Lambda(\widehat{B}_0)$ which
is the probability that the weighted random spanning
tree contains the extra edge which
equals $\text{Vol} (\widehat{\mathcal{R}}^{(\gamma)}_\Lambda\setminus\mathcal{R}^{(\gamma)}_\Lambda)/\text{Vol}({\mathcal{R}}^{(\gamma)}_\Lambda)$.
Using
\[
\text{Vol}({\mathcal{R}}^{(\gamma)}_\Lambda)= \text{det} (\widehat{\Delta}^{(\gamma)})=
\text{det}(\widehat{\Delta}^{(\gamma)}) (1+ G^\gamma_{\Lambda}(0,0))
\]
we obtain $\widehat{m}^{(\gamma)}_\Lambda(\widehat{B}_0)= G^{(\gamma)}_\Lambda(0,0)/ (1+G^{(\gamma)}_\Lambda(0,0))$
and thus \eqref{e:wavesprob}.
Consider the set
\begin{equation*}
 C_x
  = \{ \eta \in \widehat{B}_x : 2d + 1 \le \eta_0 < 2d + 1 + \gamma \}.
  \end{equation*}
By the burning algorithm,
$\widehat{m}^{(\gamma)}_\Lambda (C_x)
= \gamma \widehat{m}^{(\gamma)}_\Lambda (\widehat{B}_x)$.

Indeed we have $\eta \in\widehat{B}_x$ if and only if after adding height one at $0$ we can
topple $0$ and then carry out a wave that topples $x$. The toppling of 0 will
be possible if and only if
\[
2d + \gamma \leq \eta_0 < 2d + \gamma + 1.
\]
Once the toppling of $0$ occurred, the height $\eta_0$ plays no role in the
condition that $x$ topples. If we condition on all the heights of $\eta$
except $\eta_0$, then the conditional probability, given $\eta \in\widehat{B}_x$,
that $\eta \in C_x$ is $\gamma$ since the additional restriction is:
\[
2d + 1 \leq \eta_0 < 2d + \gamma + 1.
\]
Also note that every $\eta \in C_x$ is necessarily an intermediate configuration after
the first $\gamma$-wave of its $\gamma$-avalanche, hence we have 
\begin{equation*}
 D_x
  := \{ \eta_1 \in \widehat{\mathcal{R}}^{(\gamma)}_\Lambda :
     \widehat{a}_0^{-1} \eta_1 \in \mathcal{R}^{(\gamma)}_\Lambda,\,
     W(\eta_1) \ni x \}
  \supset C_x.
  \end{equation*}
We have
\begin{equation*}\begin{split}
 \mathrm{Vol} (\eta \in \mathcal{R}^{(\gamma)}_\Lambda : {\mathrm{Av}}_{\Lambda,0}^{(\gamma)}(\eta) \ni x )
  &\ge \mathrm{Vol} (\eta \in \mathcal{R}^{(\gamma)}_\Lambda : W^{(1)}_\Lambda(\eta) \ni x ) \\
  &= \mathrm{Vol} (\eta \in \mathcal{R}^{(\gamma)}_\Lambda : W(\widehat{a}_0 \eta) \ni x ) \\
  &= \mathrm{Vol}(D_x)
  \ge \mathrm{Vol}(C_x)
  = \gamma \mathrm{Vol} (\widehat{B}_x),
  \end{split}\end{equation*}
where in the third step, we used invariance of Lebesgue measure
under $\widehat{a}_0$. Indeed we have
\begin{eqnarray*}
\widehat{a}_0^{-1} D_x &=& \{ \widehat{a}_0^{-1}\eta_1\in \widehat{\mathcal{R}}^{(\gamma)}_\Lambda:
\widehat{a}_0^{-1}\eta_1\in {\mathcal{R}}^{(\gamma)}_\Lambda, W(\eta_1)\ni x\}
\\
&=&
\{ \eta\in \widehat{\mathcal{R}}^{(\gamma)}_\Lambda:
\eta\in {\mathcal{R}}^{(\gamma)}_\Lambda, W(\widehat{a}_0\eta)\ni x\}
\end{eqnarray*}
 This implies
\begin{equation*}\begin{split}
 m^{(\gamma)}_\Lambda ({\mathrm{Av}}_{\Lambda,0}^{(\gamma)}(\eta) \ni x)
  &\ge \gamma \widehat{m}^{(\gamma)}_\Lambda (\widehat{B}_x)
      \frac{\mathrm{Vol}(\widehat{\mathcal{R}}^{(\gamma)}_\Lambda)}{\mathrm{Vol}(\mathcal{R}^{(\gamma)}_\Lambda)}
  = \gamma \frac{G^{(\gamma)}_\Lambda(0,x)}{1 + G^{(\gamma)}_\Lambda(0,0)}
    \left( 1 + G^{(\gamma)}_\Lambda (0,0) \right) \\
  &= \gamma G^{(\gamma)}_\Lambda (0,x).
\end{split}\end{equation*}
To pass to the limit $\Lambda \to \mathbb{Z}^d$, approximate the event
$E_x = \{ {\mathrm{Av}}_{\Lambda,0}^{(\gamma)}(\eta) \ni x \}$ by
$E_{x,V} = \{ {\mathrm{Av}}_{\Lambda,0}^{(\gamma)}(\eta) \ni x,\, {\mathrm{Av}}_{\Lambda,0}^{(\gamma)}(\eta) \subseteq V \}$
with a large finite $V$. Since both $E_{x,V}$ and
$F_V = \{ {\mathrm{Av}}_{\Lambda,0}^{(\gamma)}(\eta) \not\subseteq V \}$
are local events, we have 
\begin{equation*}\begin{split}
&| m^{(\gamma)}_\Lambda (E_x) - m^{(\gamma)} (E_x) | \\
  &\qquad \le \limsup_V \Big[
    | m^{(\gamma)}_\Lambda (E_x) - m^{(\gamma)}_\Lambda (E_{x,V}) |
    + | m^{(\gamma)}_\Lambda (E_{x,V}) - m^{(\gamma)} (E_{x,V}) |  \\
  &\qquad\quad  + | m^{(\gamma)} (E_{x,V}) -  m^{(\gamma)} (E_x) | \Big] \\
  &\qquad \le \limsup_\Lambda m^{(\gamma)}_\Lambda (F_V) + 0
    + m^{(\gamma)} (F_{V}) \\
  &\qquad \le 2 \varepsilon,
  \end{split}\end{equation*}
if $V$ is large enough.

Part \textit{(b)} follows by essentially the same argument. Observe that only the dissipation
at the origin played any significant role, and adjusting this to a different
value only changes the Green's functions and the volume of recurrent
configurations by at most a constant factor depending on $d$.
\end{proof}

\begin{remark}
\begin{itemize}
\item[a)] A natural length associated to the system with dissipation
$\gamma$ is the typical diameter of a  $\gamma$-avalanche.
By Lemma \ref{masgreen}, $G^{(\gamma)}(0,x)$ decays, up to polynomial factors,
as $e^{-|x|/L(\gamma)}$, with $L(\gamma) = C/\sqrt{\gamma}$.
Thus Theorem \ref{thm:waves} and also Theorem \ref{thm:new-av-lb} below show that $1/\sqrt{\gamma}$ is the typical
avalanche-diameter as $\gamma\to 0$, and supports the idea that
in the system with dissipation $\gamma$, the ``correlation length''
scales as $1/\sqrt{\gamma}$, as $\gamma\to 0$
(compare with \cite{japon}).
\item[b)] Inequalities \eqref{doublegamma} generalize to inhomogeneous dissipation
as follows. Let $\Gamma = (\gamma_x)_{x \in \Z^d}$, be a collection of
non-negative real dissipations, and suppose that $\gamma_0 > 0$. Note that the 
Green's function $G^{(\Gamma)}(0,x)$ of the continuous-time random walk 
trapped at rate $\gamma_x$ at each lattice site $x$ is finite (even when $d = 2$). 
In particular, when $d\geq 3$ we can choose $\gamma_0 > 0$ and $\gamma_x = 0$ elsewhere
in \eqref{doublegamma}. Then the lower bound of \eqref{doublegamma} 
decays as $C\gamma_0/|x|^{d-2}$, i.e., this model has power-law decay of avalanches, 
and infinite expected avalanche area. It is an interesting problem to characterize the set
of possible inhomogeneous dissipation functions $\Gamma$ such that the expected 
area of the avalanche is infinite (i.e., the model is still ``critical'').
\end{itemize}
\end{remark}
\subsection{The toppling probability exponent in $d > 4$}
In Theorem \ref{thm:new-av-lb} below we show that for $d > 4$ the
factor $\gamma$ on the left hand side of \eqref{avesti}
can be replaced by a constant.
This implies that the probability that a site $x$ belongs to the avalanche initiated at the origin
behaves asymptotically as $G^{(\gamma)}(0,x)$ for large $x$. By taking the limit $\gamma\to 0$ one obtains
that the probability that $x$ belongs to the avalanche initiated at $0$ in the critical model behaves as
$1/|x|^{d-2}$ for large $x$.

In analogy with the one-arm exponent (connectivity exponent)
for percolation clusters \cite{Grimmett},
let us introduce the critical exponent $\theta$ by requiring that
\eqnst
{ m^{(0)} ( x \in \Av_0^{(0)}(\eta) )
  \approx \frac{c}{|x|^{d-2 + \theta}}, \quad \text{as $|x| \to \infty$.} }
Here $\approx$ means logarithmic equivalence of the two sides
(or possibly a stronger relation).
Theorem \ref{thm:new-av-lb} below shows that
when $d > 4$, the exponent takes the mean-field value $\theta = 0$.
In this context, Dhar's formula provides the mean-field bound $\theta \ge 0$.
By analogy with other critical lattice systems, it is tempting to conjecture that
below the critical dimension, i.e.~when $d < 4$, we have $\theta > 0$, and
that $\theta = 0$ with a logarithmic correction in dimension $d = 4$.
One could heuristically argue for this conjecture as follows.
When $d > 4$, avalanche clusters are tree-like \cite{Priezzhev},
and repeated topplings are not significant. In this case the expected
number of topplings gives the correct behavior of the probability that $x$ topples.
When $d < 4$, we can expect that repeated topplings are pronounced, which implies
that if $x$ topples at least once, it is likely to topple many times in the
same avalanche. In this case the probability that $x$ topples would be
significantly smaller than the expected number of topplings, and hence
$\theta > 0$.

\begin{open}
Is the above heuristic correct and $\theta > 0$ when $d = 2, 3$?
\end{open}

\noindent
We did not find any numerical work in the
literature on the exponent $\theta$ in dimensions $d = 2, 3$.
\begin{theorem}
\label{thm:new-av-lb}
When $d > 4$, there exists $c = c(d) > 0$ such that for all $\gamma \ge 0$
we have
\eqn{e:matching-lb}
{ c G^{(\gamma)} (0,x)
  \leq m^{(\gamma)} (x \in {\mathrm{Av}}_{0}^{(\gamma)} (\eta))
  \leq G^{(\gamma)} (0,x). }
\end{theorem}
\begin{proof}
The upper bound follows from Dhar's formula, so we only need to prove
the lower bound.
 The proof is divided in several steps.\\

\noindent{\bf Step 1: Reduction to a weighted spanning tree problem}

 Let $\Lambda$ be a finite volume containing $0$.
For $\eta \in \caR^{(\gamma)}_\Lambda$, let $W^\last_\Lambda(\eta)$
denote the \emph{last wave} occurring when unit height is added
at $0$ to the configuration $\eta$.

We set
$W^\last_\Lambda(\eta) = \es$, if the avalanche is empty. We have
\eqnst
{ m^{(\gamma)}_\Lambda ( x \in \Av^{(\gamma)}_{0,\Lambda}(\eta) )
  \ge m^{(\gamma)}_\Lambda ( x \in W^\last_\Lambda ). }
Recall the correspondence between waves and intermediate configurations
introduced in Section \ref{ssec:waves}.
We define a map $F : \caR^{(\gamma)}_\Lambda \to \widehat{\caR}^{(\gamma)}_\Lambda$
as follows. When the avalanche is empty, that is if
$\eta + \delta_0 \in \caR^{(\gamma)}_\Lambda$, we set
$F(\eta) = \eta + \delta_0$.
When the avalanche is non-empty, that is, if
$\eta + \delta_0 \in \widehat{\caR}^{(\gamma)}_\Lambda \setminus \caR^{(\gamma)}_\Lambda$,
we let $F(\eta)$ be the last intermediate configuration in the
avalanche, i.e., the configuration corresponding to the last wave. The mapping $F$ is one-to-one between $\caR^{(\gamma)}_\Lambda$ and
its image $F(\caR^{(\gamma)}_\Lambda)$,
and preserves Lebesgue measure between these sets. Indeed, we can subdivide
$\caR^{(\gamma)}_\Lambda$ in finitely many disjoint pieces such that on each piece
$F$ is a composition of translations (individual topplings act as translations).

For an intermediate configuration
$\zeta \in \widehat{\caR}^{(\gamma)}_\Lambda \setminus
\caR^{(\gamma)}_\Lambda$, let
$W(\zeta)$ be the wave corresponding to it.
Denote
\be\label{evector}
e := (-1,0,\dots,0) \in \Z^d.
\ee
We have that if $e \not\in W(\zeta)$, then
$(\hat{a}_0 \zeta)_0 < 2d + \gamma$, and hence $W(\zeta)$ is a last wave
(cf.\ also \cite{IKP94}).
Hence we have
\eqnspl{e:wave-lb}
{ m^{(\gamma)}_\Lambda ( x \in W^\last_\Lambda )
  &= \frac{\Vol \big( \big\{ \eta \in \caR^{(\gamma)}_\Lambda :
    x \in W^\last(\eta) \big\} \big)}{\Vol(\caR^{(\gamma)}_\Lambda)} \\
  &= \frac{\Vol \big( \big\{ \zeta \in F(\caR^{(\gamma)}_\Lambda) :
    x \in W(\zeta) \big\} \big)}{\Vol(\caR^{(\gamma)}_\Lambda)} \\
  &\ge \frac{\Vol \big( \big\{ \zeta \in \widehat{\caR}^{(\gamma)}_\Lambda \setminus \caR^{(\gamma)}_\Lambda :
    x \in W(\zeta),\, e \not\in W(\zeta) \big\} \big)}{\Vol(\caR^{(\gamma)}_\Lambda)} \\
  &\ge \frac{\Vol \big( \big\{ \zeta \in \widehat{\caR}^{(\gamma)}_\Lambda \setminus \caR^{(\gamma)}_\Lambda :
    x \in W(\zeta),\, e \not\in W(\zeta) \big\} \big)}{\Vol(\widehat{\caR}^{(\gamma)}_\Lambda)} \\
  &=: \widehat{m}^{(\gamma)}_\Lambda( E_x ). }
The event $E_x$ consists of all $\zeta \in \widehat{\caR}^{(\gamma)}_\Lambda$ such that
$2d + \gamma \le \zeta_0 < 2d + \gamma + 1$, $x \in W(\zeta)$, and $e \not \in W(\zeta)$.
Observe that $E_x$ is a union of $(\gamma,\Lambda)$-cells (here we use
a natural extension of the notion of $(\gamma,\Lambda)$-cell to
$\widehat{\caR}^{(\gamma)}_\Lambda$). Therefore, letting
$\widehat{\mu}^{(\gamma)}_\Lambda$ denote the weighted spanning tree
measure for the graph $\widehat{\Lambda}$, we have
\eqnst
{ \widehat{m}^{(\gamma)}_\Lambda( E_x )
  = \widehat{\mu}^{(\gamma)}_\Lambda ( \text{$x$ connects to $\varpi$ via the extra edge,
  and $e$ does not} ). }
We pass to the infinite volume limit (just as in the proof of Theorem \ref{thm:waves})
and get
\eqn{e:wave-tree}
{ m^{(\gamma)} ( x \in \Av^{(\gamma)}_{0}(\eta) )
  \ge \widehat{\mu}^{(\gamma)} ( \text{$x$ connects to $\varpi$ via the extra edge,
  and $e$ does not} ). }

\noindent{\bf Step 2: Reduction to a problem of two independent random walk paths}

We now use Wilson's algorithm in $\widehat{\Z^d}$ to give a lower bound on the probability
in the right hand side of \eqref{e:wave-tree}. Roughly speaking, the required event will occur,
if a network random walk started from $x$ hits $0$, then steps to $\varpi$, and
an independent network random walk started at $e$ avoids the loop-erasure
of the first path until it hits $\varpi$. Since in $d > 4$, the probability
that independent simple random walks do not intersect is bounded away from $0$,
the last event ``should not matter'', and the first two events will occur with a probability
$c G^{(\gamma)}_\Lambda(x,0)$.

In order to make the above precise, let $S$ and $S^-$ be independent
simple random walks in $\Z^d$, with $S(0) = x$ and $S^-(0) = e$. Let $T$ and
$T^-$ be independent $\mathrm{Geom}(\gamma/(2d + \gamma))$ random variables
independent of the walks. Write
\eqnsplst
{ \tau_A
  &= \inf \{ n \ge 0 : S \in A \}; \\
  \tau^-_A
  &= \inf \{ n \ge 0 : S^- \in A \} }
Write $\pi = \LE( S[0,T \wedge \tau_0 ] )$.
\eqnsplst
{ &\widehat{\mu}^{(\gamma)} ( \text{$x$ connects to $\varpi$ via the extra edge,
  and $e$ does not} ) \\
  &\qquad \ge \frac{1}{2d + \gamma + 1}
    \Pr \left( \text{$\tau_0 < T$
    and $S^-[0,T^-] \cap \pi = \es$} \right). }
Let $S^+$ be a simple random walk on $\Z^d$ with $S^+(0) = 0$,
independent of $S^-$, and let $T^+$ be its
$\mathrm{Geom}(\gamma/(2d + \gamma))$ killing time. Let
\eqnsplst
{ \tau^+_A
  &= \inf \{ n \ge 0 : S^+ \in A \} \\
  \sigma
  &= \sup \{ 0 \le n < \tau_0 : S(n) = x \} \\
  \sigma^+
  &= \sup \{ 0 \le n < \tau^+_x : S^+(n) = 0 \}. }
Conditioned on the event $\tau_0 < T$,
$\pi = \LE ( S[0,T \wedge \tau_0 ] ) = \LE ( S[\sigma,\tau_0] )$.
Due to reversibility of simple random walk and \cite[Lemma 7.2.1]{Lawler},
the latter path has the same distribution as
$\pi^+ := \LE ( S^+[\sigma^+,\tau^+_x] ) = \LE ( S^+[0,\tau^+_x] )$
conditioned on $\tau^+_x < T^+$. Therefore,
\eqnspl{e:non-intersect}
{ &\Pr \left( \text{$\tau_0 < T$
    and $S^-[0,T^-] \cap \pi = \es$} \right) \\
  &\qquad = \Pr \left( \text{$\tau^+_x < T^+$
    and $S^-[0,T^-] \cap \pi^+ = \es$} \right) \\
  &\qquad \ge \Pr \left( \text{$\tau^+_x < T^+$
    and $S^-[0,T^-] \cap S^+[0,\tau^+_x] = \es$} \right). }
Since $\Pr ( \tau^+_x < T^+ ) = \frac{G^{(\gamma)}(0,x)}{G^{(\gamma)}(0,0)}$,
we are left to show that
\eqn{e:to-show}
{ \Pr \left( \text{$\tau^+_x < T^+$
    and $S^-[0,T^-] \cap S^+[0,\tau^+_x] = \es$} \right)
  \ge c \Pr ( \tau^+_x < T^+ ). }

Notice that the event that two independent random walk paths do not meet
has positive probability when $d > 4$. However we are asking that the first path $S^+$
hits $x$, and therefore, if $x$ is large ($|x|>\gamma^{-1/2}$) the random walk survives
anomalously long, and hence it is not clear whether it behaves as an ``ordinary''
random walk. In particular this will be reflected in the proof of \eqref{e:to-show},
where a distinction between
$|x| > \gamma^{-1/2}$ and $|x| \le \gamma^{-1/2}$  has to be made.

For $\alpha \ge 0$ define the cones
\eqnsplst
{ \caH^+_\alpha
  &= \{ y \in \Z^d :
     \text{$y_1 \ge \alpha |y_j|$ for $j = 2, \dots, d$} \} \\
  \caH^-_\alpha
  &= \{ y \in \Z^d :
     \text{$y_1 < -\alpha |y_j|$ for $j = 2, \dots, d$} \}. }
We abbreviate $\caH^+ = \caH^+_1$ and $\caH^- = \caH^-_1$.
We define the events
\eqnsplst
{ H(R)
  &= \{ \text{$S^+[0,\tau^+_{B(R)^c}]
    \cap S^-[0,\tau^-_{B(R)^c}] = \es$} \} \\
  G^+(r,R)
  &= \{ \text{$S^+[\tau^+_{B(r)^c},\tau^+_{B(R)^c}] \subset \caH^+$} \} \\
  G^-(r,R)
  &= \{ \text{$S^-[\tau^-_{B(r)^c},\tau^-_{B(R)^c}] \subset \caH^-$} \}. }
In showing \eqref{e:to-show},
we assume,
without loss of
generality, that $x = (x_1, \dots, x_d)$ satisfies
\be\label{ordercoordinates}
x_1 \ge |x_j|,\  j = 2, \dots, d,
\ee

\noindent{\bf Step 3: proof of \eqref{e:to-show} for ``small'' $x$: $|x| \le \gamma^{-1/2}$.}

 Note that by transience of $S^-$ we obtain \eqref{e:to-show}
for all $x$ such that $|x| \le k_0$ for any fixed $k_0 > 0$.
So we will assume that $k_0 < |x| \le \gamma^{-1/2}$, and will choose
$k_0$ suitably in course of the proof. Define the event:
\eqnsplst
{ A
  := H(|x|/2) \cap G^+(|x|/4,|x|/2) \cap G^-(|x|/4,2|x|). }
We show that $\Pr (A) \ge c_1 > 0$.
{}From the invariance principle we derive that for all $k$ large enough and $|x|>8k$  we obtain
\eqnsplst
{ \Pr ( S^+(\tau^+_{B(|x|/4)^c}) \in \caH^+_{2} \,|\, S^+(0) = (k,0,\dots,0) )
  &\ge c \\
  \Pr ( S^-(\tau^-_{B(|x|/4)^c}) \in \caH^-_{2} \,|\, S^-(0) = (-k,0,\dots,0) )
  &\ge c. }
Hence, using the invariance principle, there exists a constant $c' > 0$
such that for all $x$  with large enough $|x|$, say $|x|> 8k$
\eqnsplst
{ \Pr ( G^+(|x|/4,|x|/2) \,|\, S^+(0) = (k,0,\dots,0) )
  &\ge c' \\
  \Pr ( G^-(|x|/4,2|x|) \,|\, S^-(0) = (-k,0,\dots,0) )
  &\ge c'. }
When $d > 4$, the probability that two independent simple
random walks starting at distance $2 k$ apart
intersect is $o(1)$ as $k \to \infty$
(see e.g.~\cite[Section 3.3]{Lawler}). Therefore we choose $k$
large enough so that
\eqnst
{ \Pr (S^+[0,\infty) \cap S^-[0,\infty) \not= \es \,|\,
      S^+(0) = (k,0,\dots,0),\, S^-(0) = (-k,0,\dots,0) )
  < c'/2, }
and set $k_0 = 8 k$. Requiring the occurrence of the event
\eqnst
{ \{ S^+(\ell) = (\ell,0,\dots,0),\, \ell = 0,\dots,k \}
  \cap \{ S^-(\ell) = (-\ell-1,0,\dots,0),\, \ell = 0,\dots,k-1 \} }
we have $\Pr(A) \ge c_1$.

Let
\eqnst
{ B
  := \{ S^-[\tau^-_{B(2|x|)^c},\infty) \cap B(3|x|/2) = \es \}. }
By the strong Markov property, transience and the invariance principle, there
exists $c_2 > 0$ such that $\Pr(A \cap B) \ge c_2$.

Let $\caD = (B(3|x|/2) \setminus B(|x|/2)) \cap \caH^+_0$. Let
\eqnst
{ C
  := \{ \text{$S^+[\tau^+_{B(|x|/2)^c},\infty)$ hits $x$ before
     exiting $\caD$} \}. }
For any $z \in \partial B(|x|/2)$ we have
\eqnsplst
{ \Pr (C \,|\, S^+(0) = z )
  &\ge \Pr ( \tau^+_{B(x,|x|/2\sqrt{d})} < \tau^+_{\caD^c} \,|\, S^+(0) = z ) \\
  &\qquad \times \min_{w \in \partial B(x,|x|/2\sqrt{d})}
      \Pr ( \tau^+_x < \tau^+_{B(x,|x|/\sqrt{d})^c} \,|\, S^+(0) = w ) \\
  &\ge c \frac{c'}{|x|^{d-2}}, }
where the second inequality follows e.g.~from \cite[Proposition 1.5.10]{Lawler}.
By the strong Markov property, we get $\Pr ( A \cap B \cap C ) \ge c_3 |x|^{d-2}$.

{}From the invariance principle we can deduce that for
all $x$ and for all $K \ge 1$ we have
\eqnst
{ \Pr ( \tau^+_x < \tau^+_{B(3|x|/2)^c},\, \tau^+_{B(2|x|)^c} \ge K |x|^2 )
  \le C_1 |x|^{2-d} e^{-c_4 K}. }
Hence for a sufficiently large $K_0$ (whose value only depends on $C_1$ and
$c_4$, and not on $x$, we have
\eqnst
{ \Pr (A \cap B \cap C \cap \{ \tau^+_{B(3|x|/2)^c} \le K_0 |x|^2 \} )
  \ge (c_3/2) |x|^{2-d}. }
Finally, note that due to $|x| \le \gamma^{-1/2}$, we have
$\Pr (T^+ \ge  K_0 |x|^2) \ge c_5 = c_5(K_0)$. We conclude the proof
by observing that on the event
\eqnst
{ A \cap B \cap C \cap \{ \tau^+_{B(3|x|/2)^c} \le K_0 |x|^2 \}
    \cap \{ T^+ > K_0 |x|^2 \} }
the required event in the left hand side of \eqref{e:to-show} occurs.\\

\noindent{\bf Step 4: Proof of \eqref{e:to-show} for``large $x$'': $|x| > \gamma^{-1/2}$}.

Let $\ell = \gamma^{-1/2}$. We need the following lemma.
\begin{lemma}
\label{lem:boundary-estimate}
There exist constants $0 < c_1 = c_1(d) < C_1 = C_1(d)$ such that for all
$y \in B(\ell/8)$ and $z \in \partial B(\ell)$ we have
\eqnst
{ \frac{c}{\ell^{d-1}}
  \le \Pr^y ( S(\tau_{B(\ell)^c}) = z,\, T > \tau_{B(\ell)^c} )
  \le \frac{C}{\ell^{d-1}}. }
\end{lemma}
\begin{proof}
Write $B = B(\ell)$. Were the event $T > \tau_{B^c}$ not present, the statement would be
\cite[Lemma 1.7.4]{Lawler}. With the event $T > \tau_{B^c}$, clearly the upper
bound still holds. In order to deduce the lower bound, it is
sufficient to show that uniformly in $y$ and $z$, the random walk
started at $y$ and conditioned on exiting $B$ at $z$ has
expected exit time at most $C' \ell^2$. Indeed, then
\eqnsplst
{ \Pr^y ( T > \tau_{B^c} \,|\, S(\tau_{B^c}) = z )
  &\ge \Pr^y ( T > 2 C' \ell^2,\, 2 C' \ell^2 \ge \tau_{B^c} \,|\,
       S(\tau_{B^c}) = z ) \\
  &\ge \left( \frac{2d}{2d + \gamma} \right)^{2 C' \ell^2} \frac{1}{2} \\
  &\ge c'. }
The conditional expected exit time can be written in the form
\eqnst
{ \sum_{x \in B} \frac{h(x)}{h(y)} G_B(y,x), }
where $h(x) = \Pr^x ( S(\tau_{B^c}) = z )$. Consider first the case
$y = 0$. We can decompose the sum and bound it above as follows:
\eqnsplst
{ &\frac{1}{h(0)} \sum_{k = 0}^\ell \sum_{k \le |x| < k+1} h(x) G_B(0,x) \\
  &\qquad \le \frac{1}{h(0)} \sum_{k = 0}^\ell \sum_{k \le |x| < k+1}
      \frac{1}{c/k^{d-1}} \Pr^0 ( S(\tau_{B(k)^c}) = x) ) h(x) G(0,x) \\
  &\qquad \le \frac{1}{h(0)} \sum_{k = 0}^\ell C k^{d-1} \sum_{k \le |x| < k+1}
      \Pr^0 ( S(\tau_{B(k)^c}) = x) h(x) k^{2-d} \\
  &\qquad = \frac{1}{h(0)} \sum_{k = 0}^\ell C k h(0) \\
  &\qquad = C' \ell^2. }
For general $y \in B(\ell/8)$, we can repeat the above argument to
bound the sum over $x : 0 \le |x - y| < \ell/2$ by $C' \ell^2$.
Due to the Harnack principle, we can also bound the sum over
$x : (3/8)\ell \le |x| \le \ell$ above by
\eqnsplst
{ \frac{1}{h(y)} \sum_{k = (3/8) \ell}^\ell \ \sum_{k \le |x| < k+1} h(x) G_B(y,x)
  &\le \frac{C}{h(0)} \sum_{k = (3/8) \ell}^\ell \ \sum_{k \le |x| < k+1} h(x) G_B(0,x) \\
  &\le C' \ell^2. }
The two ranges of $x$'s cover all of $B$, so the claim follows.
\end{proof}
Let us now fix $\delta > 0$ in such a way that
for $y \in \partial B(\ell/16) \cap \caH^+$ we have
\eqnst
{ \Pr^y ( \tau_{B(\delta \ell)} < \infty )
  \le \frac{c_1}{2 C_1}, }
where $c_1$ and $C_1$ are the constants in Lemma \ref{lem:boundary-estimate}.
This is possible, as long as $\ell$ is sufficiently large, since
the probability scales as $\delta^{d-2}$. Taking $\ell$ sufficiently
large requires us to restrict to $\gamma \le \text{ some } \gamma_0$.
We will comment on the case $\gamma > \gamma_0$ at the end.

Define the event
\eqnst
{ A
  := H(\delta \ell)
     \cap G^+(\delta \ell/2, \delta \ell)
     \cap G^-(\delta \ell/2, \delta \ell), }
that satisfies $\Pr (A) \ge c_2$. For a sufficiently large $K_0$ we have
$\Pr ( \tau^-_{B(\delta \ell)^c} < \frac{1}{K_0} \ell^2 ) < c_2 / 4$
and $\Pr ( \tau^+_{B(\delta \ell)^c} > K_0 \ell^2 ) < c_2 / 4$. Hence
with some $c_3 = c_3(K_0) > 0$ we have
\eqn{e:A-lb}
{ \Pr ( A \cap \{ T^- < \tau^-_{B(\delta \ell)^c} \} \cap \{ T^+ > \tau^-_{B(\delta \ell)^c} \} )
  \ge \frac{c_2}{2} c_3. }
Denote by $\caD = B(\delta \ell/2) \cup (\caH^- \cap B(\delta \ell))$.
Given $z \in \partial B(\delta \ell) \cap \caH^+$, define the event
\eqnst
{ C_z
  := \{ S^+(0) = z,\, \tau^+_x < \tau^+_\caD \wedge T^+ \}. }
We show that
\eqn{e:worst-case}
{ \inf_{z \in \partial B(\delta \ell) \cap \caH^+} \Pr^z (C_z)
  \ge c \Pr^0 ( \tau^+_x < T^+ ), }
which implies the claim \eqref{e:to-show} in light of \eqref{e:A-lb} and
the strong Markov property of $S^+$.

First, it is clear from the invariance principle and $\gamma \asymp 1/\ell^2$ that
the killed random walk started at $z \in \partial B(\delta \ell) \cap \caH^+$
has probability bounded away from $0$ to reach $\partial B(\ell/16) \cap \caH^+$
without hitting $\caD$. Hence \eqref{e:worst-case} will be proved once we show
\eqn{e:worst-case2}
{ \inf_{y \in \partial B(\ell/16) \cap \caH^+} \Pr^y (C_y)
  \ge c \Pr^0 ( \tau^+_x < T^+ ). }
Using Lemma \ref{lem:boundary-estimate} we write
\eqnspl{e:upper-bound-tau+x}
{ \Pr^0 (\tau^+_x < T^+)
  &= \sum_{z \in \partial B(\ell)} \Pr^0 ( S^+(\tau^+_{B(\ell)^c}) = z,\,
    T^+ > \tau^+_{B(\ell)^c} ) \Pr^z ( \tau^+_x < T^+ ) \\
  &\le \frac{C_1}{\ell^{d-1}} \sum_{z \in \partial B(\ell)} \Pr^z ( \tau^+_x < T^+ ). }
On the other hand, for $y \in \partial B(\ell/16) \cap \caH^+$ we have
\eqnsplst
{ &\Pr^y (\tau^+_x < \tau^+_\caD \wedge T^+ ) \\
  &\qquad \ge \sum_{z \in \partial B(\ell)} \Pr^y ( S(\tau^+_{B(\ell)^c}) = z,\, T^+ > \tau^+_{B(\ell)^c} )
      \Pr^z ( \tau^+_x < T^+ ) \\
  &\qquad\qquad - \Pr^y ( \tau^+_\caD < \infty ) \max_{y' \in \caD}
      \sum_{z \in \partial B(\ell)} \Pr^{y'} ( S(\tau^+_{B(\ell)^c}) = z,\, T^+ > \tau^+_{B(\ell)^c} )
      \Pr^z ( \tau^+_x < T^+ ) \\
  &\qquad \ge \frac{c_1}{\ell^{d-1}} \sum_{z \in \partial B(\ell)} \Pr^z ( \tau^+_x < T^+ )
      - \frac{c_1}{2 C_1} \frac{C_1}{\ell^{d-1}}
      \sum_{z \in \partial B(\ell)} \Pr^z ( \tau^+_x < T^+ ) \\
  &\qquad = \frac{c_1}{2 \ell^{d-1}} \sum_{z \in \partial B(\ell)} \Pr^z ( \tau^+_x < T^+ ). }
Together with \eqref{e:upper-bound-tau+x} this yields the
required statement \eqref{e:worst-case2}, and completes the proof
in the case $0 < \gamma \le \gamma_0$.

When $\gamma_0 < \gamma < 1$, the statement of the theorem is implied by \eqref{avesti}.
When $\gamma \ge 1$, observe that every avalanche consists of only one wave,
so $m^{(\gamma)}( x \in \Av^{(\gamma)}_0(\eta) ) = G^{(\gamma)}(0,x)$.
\end{proof}

\section{Zero dissipation limit of the stationary processes}
\label{sec:zero-dissip-lim}
We now consider the infinite volume dynamics in the case $\gamma = 0$,  and show that
it is the limit of the infinite volume dynamics with dissipation $\gamma$ 
when $\gamma\downarrow 0$.

For the   abelian sandpile
model, this infinite volume dynamics  was constructed in \cite{jr},
in dimensions $d \ge 3$. We recall the main steps of this construction, and
 indicate  how it applies to the  abelian avalanche
model.

Recall Remark \ref{rem:uniform}, and note that when $\gamma = 0$,
the dynamics preserves the fractional part of each coordinate.
\bl
\begin{itemize}
\item[(i)] For $m^{(0)}$-a.e.~$\eta_0$, the configuration
$\eta_0 + \delta_x$ is $0$-stabilizable for all $x \in \mathbb{Z}^d$.
\item[(ii)] For all $x \in \mathbb{Z}^d$, $a^{(0)}_x$ leaves $m^{(0)}$ invariant.
\item[(iii)] Assume that ${\varphi}$ satisfies
$\sum_{x \in \mathbb{Z}^d} {\varphi}(x) G^{(0)}(x,0) < \infty$. Then
$m^{(0)} \otimes {\mathbb P}$-a.s., the limit
\begin{equation*}
 \eta_t
  = \lim_V \prod_{x \in V}
    \left( a^{(0)}_x \right)^{N^{\varphi}_{x,t}} \eta_0
    \end{equation*}
exists, and equals ${\mathcal S}^{(0)} ( \eta_0 + N^{\varphi}_t )$.
\end{itemize}
\el
\bpr
\begin{itemize}
\item[\textit{(i)}]
This is because $\eta_0 + \delta_x$ is $0$-stabilizable if and only if
its image under $\psi$ is stabilizable. It was shown in
\cite{jr} that for $\nu^{(0)}$-a.e.~configuration $\xi$,
for all $x \in \mathbb{Z}^d$, $\xi + \delta_x$ is stabilizable.
This implies {\em (i)}.
\item[\textit{(ii)}]
It was shown in \cite{jr} that the discrete addition
operators leave $\nu^{(0)}$ invariant. Denoting the
discrete stabilization operator by ${\mathcal S}^{\mathrm{discr}}$, we have
\begin{equation}\label{e:psidiscr}
 \psi({\mathcal S}^{(0)}(\eta + \delta_x))
  = {\mathcal S}^{\mathrm{discr}}(\psi(\eta) + \delta_x),
  \end{equation}
which implies {\em (ii)}.
\item[\textit{(iii)}]
Again, this was shown in the discrete case,
and \eqref{e:psidiscr} implies it for the continuous case,
  that is the abelian avalanche model.
\end{itemize}
\epr
{}From now on, we denote $m:= m^{(0)}$, $G(\cdot,\cdot):= G^{(0)}(\cdot,\cdot)$,
$a_x:= a_x^{(0)}$ and ${\mathrm{Av}}_x (\eta):= {\mathrm{Av}}^{(0)}_x (\eta)$.

Let ${\varphi}$ be an addition rate such that
\begin{equation}\label{kolibri}
 \sum_{x\in \mathbb{Z}^d} {\varphi}(x) G(x,y) <\infty
\end{equation}
for all $y\in\mathbb{Z}^d$.
Let $\eta^{(\gamma)}_{t}$ denote the stationary process
obtained when starting from $\eta^{(\gamma)}_0:=\eta^{(\gamma)}$
distributed according to $m^{(\gamma)}$, making additions
according to independent Poisson processes with
rate ${\varphi}(x)$ at $x\in \mathbb{Z}^d$, and
stabilizing with dissipation $\gamma$. Similarly, let
$\eta_t$ denote the process starting from $\eta_0=\eta$
distributed according to $m$, making additions
according to independent Poisson processes with
rate ${\varphi}(x)$ at $x\in \mathbb{Z}^d$, and
stabilizing without dissipation, i.e., with $\gamma=0$.

We write $D[0,1]$ for the space of c\`adl\`ag functions
in ${\mathcal X}$ endowed with the Skorokhod topology.
\begin{theorem}
\label{thm:converge}
When $d \geq 3$, and condition \eqref{kolibri} is satisfied,
the process $\eta^{(\gamma)}_{t}$ converges weakly in $D[0,1]$
to $\eta_t$.
\end{theorem}
\subsection{ Convergence of the addition operator}
We need the following two lemmas.
\begin{lemma}
\label{lem:avalcont}
Suppose that $0 \le \gamma' < \gamma$. Then \\
(i) $N^{(\gamma)}(\eta) \le N^{(\gamma')}(\eta)$;\\
(ii) ${\mathrm{Av}}_0^{(\gamma)}(\eta) \subseteq {\mathrm{Av}}_0(\eta)$.
\end{lemma}
\begin{proof}
{\em (i)} Consider a $\gamma$-exhaustive sequence
$y_1, y_2, \dots$ of $\gamma$-legal $\gamma$-topplings
for $\eta$. We show that the
same sequence is $\gamma'$-legal for $\eta$.
Since $y_1$ is $\gamma$-unstable in $\eta$,
it is also $\gamma'$-unstable. After its
$\gamma$-toppling let us add height
$\gamma - \gamma' > 0$ at $y_1$. This has the
same effect as if we performed a
$\gamma'$-toppling, and the added extra height
does not affect $\gamma$-legality of the sequence.
Adding similarly after each $\gamma$-toppling
shows that $y_1, y_2, \dots$ is $\gamma'$-legal,
and {\em (i)} follows by the remarks preceding
Lemma \ref{lem:stabilize}.

{\em (ii)} This follows immediately from {\em (i)}.
\end{proof}
\begin{lemma}
\label{lem:gammaconv}
Suppose that $\gamma_n \downarrow 0$. Then
for $m$-a.e.~$\eta$,
\begin{equation}
\label{brom}
  \lim_{n \to \infty} a^{(\gamma_n)}_0 (\eta)
  = a^{(0)}_0 (\eta).
\end{equation}
\end{lemma}
\begin{proof}
We have
\begin{equation}\label{avfin}
m ( |{\mathrm{Av}}_0(\eta)| < \infty ) = 1;
\end{equation}
see \cite[Theorem 3.11]{jr}. Also,
$m ( \eta_x \not \in \{0, 1, \dots, 2d-1 \},\, x \in \mathbb{Z}^d )
= 1$. On the intersection of the two events, there
exists a random $\alpha > 0$ (depending on $\eta$),
such that
\begin{equation}\label{e:restriction}
 \text{$\eta_x \not\in [j, j+\alpha]$, for all
  $x \in {\mathrm{Av}}_0(\eta)$ and all $j = 0, 1, 2, \dots$}
  \end{equation}
Let $M = \max \{ (N^{(0)}\eta)_x : x \in {\mathrm{Av}}_0(\eta) \}$.
We claim that when $\gamma_n < \alpha (M+1)^{-1}$, then the
$\gamma_n$-topplings satisfy $N^{(\gamma_n)}\eta = N^{(0)}\eta$.
Observe that the first toppling $T^{(0)}_0$ is applied
if and only if $\eta_0 \in (2d-1+\alpha,2d)$. This means
that $T^{(\gamma_n)}_0$ is also applied. After toppling,
each $x \not=0$ in the avalanche cluster still satisfies
\eqref{e:restriction} (since the height changed by an
integer amount). On the other hand, at $x=0$, the
condition is weakened to $\eta_0 \notin [j,j+\alpha-\gamma_n]$,
which implies $\eta_0 \not\in [j,j+\alpha M (M+1)^{-1}]$.
Continuing inductively, we get that all topplings in the
computation of ${\mathcal S}^{(0)}(\eta + \delta_0)$ occur
under the $\gamma_n$-stabilization of $\eta + \delta_0$.
Due to Lemma \ref{lem:avalcont} {\em (i)}, this proves the claim.
Together with $\Delta^{(\gamma_n)} \to \Delta^{(0)}$
this implies the statement.
\end{proof}
\begin{proof}[Proof of Theorem \ref{thm:converge}]
Recall the definition \eqref{e:dist} of the metric on
$[0, 2d+\gamma)^{\mathbb{Z}^d}$, that is
\[
 {\mathrm{dist}} (\eta,\zeta)
 = \sum_{x\in\mathbb{Z}^d} 2^{-|x|} \min\{ |\eta_x-\zeta_x|, 1 \}.
\]
As a first step, for a given $\delta >0$, we will prove that
for all $\varepsilon>0$, there
exists a coupling ${\mathcal M}^{(\gamma)}$ of $m^{(\gamma)}$ and $m$ such that,
in this coupling,
with probability at least $1-\varepsilon$,
\begin{equation}\label{distance}
 {\mathrm{dist}}(a_0^{(\gamma)} (\eta^{(\gamma)}), a_0^{(\gamma)} (\eta^{(0)}))
 \leq \delta.
\end{equation}
This will allow us to deal with the convergence of the processes
for addition rates with finite support.

First, for given $\delta,\varepsilon >0$, we choose $V$ large
enough such that if $\eta,\zeta$ agree on $V$, then
${\mathrm{dist}}(\eta,\zeta)<\delta$, and such that
\begin{equation}\label{mava}
  m \left(\overline{ {\mathrm{Av}}_0(\eta)} \not\subseteq V\right)
  \leq \varepsilon.
\end{equation}
Such a choice of $V$
is possible, since avalanches
are finite with $m$-probability one by \eqref{avfin}.
By Lemma \ref{lem:avalcont}{\em (ii)},
we then have the same estimate \eqref{mava} for
${\mathrm{Av}}^{(\gamma)}_0(\eta)$.

Since $m^{(\gamma)}\to m$ weakly, and restrictions to
finite volumes $\Lambda\subseteq\mathbb{Z}^d$ of $m^{(\gamma)}$ and $m$ are
absolutely continuous with respect to Lebesgue measure,
by \cite[Proposition 1]{thor}
we have the existence of $\gamma_0 > 0$ such that
for all $\gamma<\gamma_0$ there exists a coupling ${\mathcal M}^{(\gamma)}$ of
$m^{(\gamma)}$ and $m$ such that
\begin{equation}\label{coupi}
{\mathcal M}^{(\gamma)} (\eta^{(\gamma)}_x= \eta^{(0)}_x,\ \forall x\in V) \geq 1-\varepsilon.
\end{equation}
In the coupling ${\mathcal M}^{(\gamma)}$ we then have, by \eqref{mava}
\[
 \left(a_0^{(\gamma)} (\eta^{(\gamma)})\right)_y
 = \left(a_0^{(\gamma)} (\eta^{(0)})\right)_y
\]
for all $y\in V$, with probability at least $1-3\varepsilon$.
Therefore, the probability that
the distance ${\mathrm{dist}}(a_0^{(\gamma)} (\eta^{(\gamma)}), a_0^{(\gamma)} (\eta^{(0)}))$ is
larger that $\delta$ is smaller than $3  \varepsilon$.

So far, we can conclude that
 $a_0^{(\gamma)} (\eta^{(\gamma)})\to a_0^{(\gamma)} (\eta^{(0)})$
weakly as $\gamma\to 0$. By Lemma \ref{lem:gammaconv},
$a_0^{(\gamma)} (\eta) \to a_0 (\eta)$ for $m$-a.e.~$\eta$.
Hence we have
\begin{equation}
\label{weao}
  a_0^{(\gamma)} (\eta^{(\gamma)}) \to a_0 (\eta^{(0)})
\end{equation}
weakly as $\gamma\to 0$.
Analogously, using finiteness of avalanches,
we conclude that for any finite set $B\subseteq\mathbb{Z}^d$, and
natural numbers $n_x, x\in B$, we have
\begin{equation}
  \prod_{x\in B} (a_x^{(\gamma)})^{n_x} (\eta^{(\gamma)})
  \to \prod_{x\in B} a_x^{n_x} (\eta^{(0)})
\end{equation}
weakly, as $\gamma\to 0$.
Therefore, we have convergence of the processes
$\eta^{(\gamma)}_t\to\eta^{(0)}_t$ for
addition rates with finite support, i.e., such that ${\varphi}(x)=0$
for $x\not\in D$ with $D\subseteq\mathbb{Z}^d$ finite.
\subsection{ Convergence of the semigroups for general addition rates}
The next step is to pass to general addition rates;
we use the convergence of semigroups argument,
as in the proof of \cite[Proposition 4.1]{mrs1}.
Let $S^{{\varphi},\gamma}_t$ denote the semigroup of
the process $\eta^{(\gamma)}_t$ with addition rate ${\varphi}$,
and $S^{\varphi}_t$ the semigroup of the process
$\eta_t$ (with zero dissipation) with addition rate ${\varphi}$. Both semigroups
are well-defined as long as
${\varphi}$ has finite support.
For a local function $f$, with dependence set $D_f$,
and for addition rates ${\varphi},{\varphi}'$ of finite support,
we have, similarly to the estimate (51) in
the proof of  \cite[Proposition 4.1]{mrs1}
\begin{equation*}\begin{split}
{\mathbb E}_{m^{(\gamma)}}|S^{{\varphi},\gamma}_t (f) - S^{{\varphi}',\gamma}_t (f)|
  &\leq C t \sum_{x\in \overline{D_f}}
      \sum_{y\in\mathbb{Z}^d} G^{(\gamma)} (x,y)|{\varphi}(y)-{\varphi}'(y)| \\
  &\leq C t \sum_{x\in \overline{D_f}}
      \sum_{y\in\mathbb{Z}^d} G (x,y)|{\varphi}(y)-{\varphi}'(y)|;
      \end{split}\end{equation*}
and
\[
  {\mathbb E}_{m}|S^{{\varphi}}_t (f) - S^{{\varphi}'}_t (f)|
  \leq C t \sum_{x\in \overline{D_f}}
      \sum_{y\in\mathbb{Z}^d} G (x,y)|{\varphi}(y)-{\varphi}'(y)|.
\]
Note that this estimate in \cite{mrs1} is given in the
context of   the abelian sandpile model, that is   a model with
discrete heights and no dissipation. However, it is based only on the
estimate \eqref{markest}
for the numbers of topplings, which is valid
for the   abelian avalanche
 model, and therefore it extends directly
to the abelian avalanche
  model.

Hence, if
for a sequence ${\varphi}^{(n)}$ of addition rates
of finite support, for an addition rate ${\varphi}$
(not necessarily of finite support)
and for all $x\in\mathbb{Z}^d$
\[
\sum_{y\in\mathbb{Z}^d} G (x,y)|{\varphi}(y)-{\varphi}^{(n)}(y)|\to 0
\]
as $n\to\infty$, we conclude for all local $f$, $\gamma\geq 0$,
$S^{{\varphi}^{(n)},\gamma}_t (f)$ is a Cauchy sequence
in $L^1 (m^{(\gamma)})$, and hence converges to
$\Psi:= S^{{\varphi},\gamma}_t (f)$ for
all $\gamma\geq 0$, as $n\to\infty$. This
semigroup $S^{{\varphi},\gamma}_t $ then defines
a corresponding stationary Markov process $\eta^{{\varphi},(\gamma)}_t$,
and $\eta^{\varphi}_t$ (for $\gamma=0$).
As the convergence of the semigroups implies the convergence
of the finite dimensional distributions
of the corresponding stationary processes, we conclude for all $\gamma>0$,
\begin{equation}\label{aa}
 \eta^{{\varphi}^{(n)}, \gamma}_t\to \eta^{{\varphi},\gamma}_t
\end{equation}
and
\begin{equation}\label{bee}
 \eta^{{\varphi}^{(n)}}_t\to \eta^{{\varphi}}_t
\end{equation}
as $n\to\infty$
in the sense of convergence of finite dimensional distributions.
Therefore, if ${\varphi}$ satisfies \eqref{kolibri},
let ${\varphi}^{(n)}$ denote
${\varphi}^{(n)}(x) = {\varphi}(x) I(x\in [-n,n]^d)$, then we
have for all $n\in{\mathbb N}$
\begin{equation}\label{cee}
\eta^{{\varphi}^{(n)}, \gamma}_t\to \eta^{{\varphi}^{(n)}}_t
\end{equation}
as $\gamma\to 0$. Combination of \eqref{aa},\eqref{bee},\eqref{cee}
together with a three epsilon argument then concludes the convergence
of $\{\eta^{{\varphi}, \gamma}_t:t\geq 0\}$ to $\{\eta^{{\varphi},0}_t:t\geq 0\}$
in the sense of convergence
of finite dimensional distributions, as $\gamma\downarrow 0$.
\subsection{Tightness}
To finish the proof, we have to show that the processes
$\{ \eta^{{\varphi},\gamma}_t:t\geq 0\}$ form a tight family if
${\varphi}$ satisfies \eqref{kolibri}. By definition of the product distance
between configurations, this reduces to showing that
for all $x\in\mathbb{Z}^d$, and for all $\varepsilon >0$,
\begin{equation}\label{cobra}
{\mathbb{P}}\left(\sup_{0\leq s\leq t\leq\delta} |\eta^{{\varphi},\gamma}_s(x)-\eta^{{\varphi},\gamma}_t(x)|
  \geq \varepsilon\right)\leq C_\varepsilon \delta
\end{equation}
where the constant $C_\varepsilon$ only depends on $\varepsilon$.
Let $L= \sum_{y\in\mathbb{Z}^d} {\varphi}(y) (a_y^{(\gamma)}-I)$
denote the generator of the process
$\eta^{{\varphi},\gamma}_t$, and  $f_x (\eta)=\eta(x)$. We have that
\begin{equation}\label{marti}
\eta^{{\varphi},\gamma}_t(x)-\eta^{{\varphi},\gamma}_0(x) -
\int_0^t \sum_{y\in\mathbb{Z}^d}{\varphi}(y)
\left( (a_y^{(\gamma)} \eta_s )(x)-\eta_s(x)\right) ds= M_t
\end{equation}
is a martingale with quadratic variation
\begin{eqnarray}\label{kwadvar}
<M_t, M_t>&=&\int_0^t \left(L (f^2_x)-2f_xL(f_x)\right)(\eta_s) ds
\\ \label{uitkwad}
&=& \int_0^t \left(\sum_{y\in\mathbb{Z}^d} {\varphi}(y)
  \left((a^{(\gamma)}_y \eta_s)(x)\right)^2- \eta_s(x)^2)
\right.\\ \nonumber &&+
\left.\eta_s(x) \sum_{y\in\mathbb{Z}^d} {\varphi}(y)
  \left((a^{(\gamma)}_y \eta_s)(x)\right) - \eta_s(x))\right)ds.
\end{eqnarray}
Using that the heights are uniformly bounded by a constant, we estimate
\begin{equation}\label{kwadest}
  |\left(L (f^2_x)-2f_xL(f_x)\right)(\eta_s)|
  \leq C\sum_{y\in\mathbb{Z}^d} {\varphi}(y) I
    \left((a^{(\gamma)}_y \eta_s)(x)\not=\eta_s (x)\right).
\end{equation}
Now, since
\[
{\mathbb{P}}\left((a^{(\gamma)}_y \eta_s)(x)\not=\eta_s (x)\right)\leq
\sum_{z\sim x} G^{(\gamma)}(y,z),
\]
the stationary expectation of $<M_t,M_t>$ is bounded by
\begin{equation}\label{kwadbound}
{\mathbb E}(<M_t,M_t>)\leq Ct\sum_{y\in\mathbb{Z}^d}\sum_{z\sim x} {\varphi} (y) G(y,z) < C_1 t,
\end{equation}
where \eqref{kolibri} gives the final bound.
Similarly,
\begin{equation}\label{intest}
{\mathbb E}\left|\int_s^t \sum_{y\in\mathbb{Z}^d}{\varphi}(y)\left(
(a_y^{(\gamma)}\eta_r )(x)-\eta_r(x)\right) dr\right|
\leq C_2 (t-s).
\end{equation}
Then use Markov's and Doob's inequality
to conclude \eqref{cobra}:
\begin{eqnarray}
&&{\mathbb{P}}\left(\sup_{0\leq s\leq t\leq\delta}
  |\eta^{{\varphi},\gamma}_s(x)-\eta^{{\varphi},\gamma}_t(x)|\geq \varepsilon\right)
\nonumber\\
&\leq &
{\mathbb{P}}\left(\int_0^\delta |L f_x (\eta_r)| dr\geq \frac\varepsilon{2}\right)
+
{\mathbb{P}}\left(\sup_{0\leq s\leq t\leq \delta}|M_t-M_s|\geq \frac\varepsilon{2}\right)
\nonumber\\
&\leq &\frac{C'_1 \delta}{\varepsilon}  + \frac{C'_2 \delta}{\varepsilon^2}.\nonumber
\end{eqnarray}
\end{proof}

{\bf Acknowledgements.} We thank Hermann Thorisson for
indicating us reference \cite{thor}. E.S. was supported by grants ANR-07-BLAN-0230, ANR-2010-BLAN-0108.
For financial support and
hospitality, we thank Carleton University, Leiden University,
MAP5 lab at Universit\'e Paris Descartes,
Nijmegen University, Delft University and Centre Emile Borel of Institut Henri Poincar\'e
 (part of this work was done during the
semester ``Interacting Particle Systems, Statistical Mechanics
and Probability Theory'').

\end{document}